\documentclass[11pt]{amsart}
\usepackage{amssymb,latexsym,amsmath}
\usepackage{bm}
\usepackage{upgreek}
\input{xypic}
\xyoption{all}

\begin{document}


\title[]{Arithmeticity for periods of automorphic forms}

\author{\bf Wee Teck  Gan \ \ \and \ \ A. Raghuram}

\address{Wee Teck Gan: Department of Mathematics, National University of Singapore, 10 Lower Kent Ridge Road
Singapore 119076}
\email{matgwt@nus.edu.sg}

\address{A. Raghuram: Indian Institute of Science Education and Research (IISER), First floor, Central Tower, Sai Trinity Building Garware Circle, Sutarwadi, Pashan Pune, Maharashtra 411021, India.}

\email{raghuram@iiser.pune.ac.in}

\thanks{W.T.G is partially supported by a startup grant at the National University of Singapore. A.R. is partially supported by the National Science Foundation (NSF), award number DMS-0856113, and an Alexander von Humboldt Research Fellowship.}

\date{\today}
\subjclass{ }

\maketitle


\def\g{\mathfrak{g}}
\def\k{\mathfrak{k}}
\def\z{\mathfrak{z}}
\def\h{{\mathfrak h}}
\def\gl{\mathfrak{gl}}

\def\Ext{{\rm Ext}}
\def\Hom{{\rm Hom}}
\def\Ind{{\rm Ind}}

\def\GL{{\rm GL}}
\def\SL{{\rm SL}}
\def\SO{{\rm SO}}
\def\O{{\rm O}}
\def\Sp{{\rm Sp}}
\def\GSpin{{\rm GSpin}}
\def\U{{\rm U}}

\def\R{\mathbb{R}}
\def\C{\mathbb{C}}
\def\Z{\mathbb{Z}}
\def\Q{\mathbb{Q}}
\def\A{\mathbb{A}}
\def\S{\mathcal{S}}
\def\E{\mathcal{E}}
\def\W{\mathcal{W}}
\def\G{\mathcal{G}}

\def\w{\wedge}

\def\autc{{\rm Aut}({\mathbb C})}
\def\bs{\backslash}
\def\Cat{\mathcal{C}}
\def\HC{{\rm HC}}
\def\HCat{\Cat^\HC}
\def\proj{{\rm proj}}

\def\to{\rightarrow}
\def\To{\longrightarrow}

\def\1{1\!\!1}
\def\dim{{\rm dim}}

\def\th{^{\rm th}}
\def\isom{\approx}

\def\CE{\mathcal{C}\mathcal{E}}



\makeatletter
\newcommand{\bfgreek}[1]{\bm{\@nameuse{up#1}}}
\makeatother
\newcommand{\micamp}{\mbox{$\bfgreek{mu}$A}}
\newcommand{\micronup}{\mbox{$\bfgreek{mu}$m}}
\newcommand{\Picamp}{\mbox{$\bfgreek{Pi}$A}}
\newcommand{\Picronup}{\mbox{$\bfgreek{Pi}$m}}

\numberwithin{equation}{section}
\newtheorem{thm}[equation]{Theorem}
\newtheorem{cor}[equation]{Corollary}
\newtheorem{lemma}[equation]{Lemma}
\newtheorem{prop}[equation]{Proposition}
\newtheorem{con}[equation]{Conjecture}
\newtheorem{ass}[equation]{Assumption}
\newtheorem{defn}[equation]{Definition}
\newtheorem{rem}[equation]{Remark}
\newtheorem{exer}[equation]{Exercise}
\newtheorem{exam}[equation]{Example}


\section{Introduction}

Let $G$ be a connected reductive algebraic group over a number field $F$, and 
let $(\pi,V_\pi)$ be a cuspidal automorphic representation of $G(\A_F)$. Let $H$ be an algebraic $F$-subgroup of $G$, and let $\chi$ be an automorphic  character of $H(\A_F).$  We say that $\pi$ has a non-vanishing  $(H,\chi)$-period if the functional 
\begin{equation}
\label{eqn:l-chi}
\phi \mapsto \ell_\chi(\phi) := \int_{[H]} \chi(h)^{-1}\phi(h)\, dh, \ \ \ \phi \in V_\pi, 
\end{equation}
is nonzero, where 
$[H] := H(F)\backslash H(\A_F)$ or sometimes $[H] := Z_G(\A_F)H(F)\backslash H(\A_F).$
 Let us now suppose that we are in an arithmetic situation, in as much as that we can talk of the automorphic representation ${}^{\sigma}\pi$ for any $\sigma \in \autc.$ For example, if $\pi$ is a cohomological cuspidal automorphic representation of $\GL_n(\A_F)$ then, by a result of Clozel (see Theorem~\ref{thm:clozel} below), 
we know that so is ${}^{\sigma}\pi$. 
In this paper we study, mostly by the way of presenting a lot of examples, the dictum:
{\it $\pi$ has a non-vanishing  $(H,\chi)$-period if and only if 
${}^{\sigma}\pi$ has a non-vanishing  $(H, {}^{\sigma}\chi)$-period.} It is this dictum that we call `arithmeticity for periods of  automorphic forms.'  
\vskip 5pt

Let us remind the reader that automorphisms of $\C$, with the exceptions of the identity automorphism and complex-conjugation, are discontinuous; in particular, it is almost never the case that $\sigma(\ell_\chi(\phi)) = \ell_{\sigma\circ \chi}(\sigma(\phi)).$ So one cannot naively take the $\sigma$ inside the integral sign. In every example that we study, the dictum holds, and the argument is always indirect via some characterization of existence of such periods. 
\vskip 5pt

Let us also observe at the outset that the problem is a distinctly global problem. The corresponding local problem, at any finite place $v$, is trivial: if 
$\pi_v$  is $(H_v,\chi_v)$-distinguished, i.e., there exists a nonzero functional $\ell : \pi_v \to \C$ such that $\ell(\pi_v(h)v) = \chi_v(h)\ell(v)$ for all $h \in H(F_v)$. For any $\sigma \in \autc,$ it is easy to see that $\sigma \circ \ell$ gives a 
$(H_v, {}^\sigma \chi_v)$-distinguishing functional for the conjugated representation ${}^\sigma\pi_v.$ 
\vskip 5pt

The above local observation says that the problem of arithmeticity of automorphic periods is a consequence of a positive solution of  the classical local-to-global problem: `If $\tau$ is an automorphic representation, and suppose at every place $v$,  $\tau_v$ is $(H_v,\chi_v)$-distinguished, then does $\tau$ have a nonzero global $(H,\chi)$-period?' Here, take $\tau$ to be ${}^\sigma\pi$.

\medskip

Given a cuspidal automorphic representation $\pi$ of $G(\A),$ and a $\sigma \in \autc$ we need to discuss when the representation ${}^\sigma\pi$ makes sense. This will be possible when the representation $\pi$ contributes to the cohomology of a locally symmetric space of $G$ with coefficients in a sheaf attached to a finite-dimensional coefficient system. In Section~\ref{sec:cohomology} we briefly discuss the appropriate cohomological preliminaries needed to talk about the Galois-conjugated representation ${}^\sigma\pi$ for a general $G$, and in Section~\ref{sec:GL(N)} we explicate the case of $\GL(n)$ and recall the basic Theorem \ref{thm:clozel} due to Clozel which says that ${}^\sigma \pi$ is also a cuspidal cohomological representation. In Theorem \ref{thm:classical}, we give the natural generalization of Clozel's theorem to certain classical groups, exploiting the recent results of Arthur \cite{arthur-book}. 

\vskip 5pt

In Section~\ref{sec:gl2}, we begin by looking at two of the easiest nontrivial examples when the ambient group $G = \GL_2/F$. In particular, in the $\GL_2$ context, we look at the question of arithmeticity for Whittaker periods which boils down to every cuspidal automorphic representation being globally generic and that the space of cuspidal cohomology having a rational structure; indeed, the same ingredients give arithmeticity for Whittaker periods when $G = \GL_n/F$. 
The other $\GL_2$ example we analyze is $(\GL_1, \chi)$-periods for Hecke characters $\chi$; here, 
arithmeticity is a consequence of   Manin's and Shimura's classical algebraicity results on the critical values of $L$-functions for $\GL_2$. 
\vskip 5pt

In the rest of the paper we analyze the following situations for arithmeticity problems, which  are various generalizations of the $\GL_2$ cases considered in Section~\ref{sec:gl2}:
\begin{enumerate}
\item Shalika period integrals for representations of $\GL_{2n}$. See Theorem~\ref{thm:arithmeticity-shalika}. The nonvanishing of Shalika period integrals is characterized in terms of functorial transfers from $\GSpin(2n+1)$.   

\smallskip

\item  $\GL(n)/F$-periods for representations of $\GL(n)/E$, for a quadratic extension $E/F.$ See Theorem \ref{thm:arithmeticity-unitary}.
The nonvanishing of such period integrals is characterized in terms of functorial transfers from the unitary groups $\U(n)$.

\smallskip 

\item $\GL(n-1)$-periods for representations of $\GL(n) \times \GL(n-1).$ 
See Theorem~\ref{thm:ggp-gln}. These are the Gross-Prasad periods for the general linear groups.  

\smallskip

\item $\GL(n) \times \GL(n)$-periods for representations of $\GL(2n)$ over a totally real field.  
See Theorem~\ref{thm:gl2n}. 

\smallskip

\item  Whittaker and Gross-Prasad periods for classical groups. See Theorem~\ref{thm:whittaker-orthogonal}. 

 \end{enumerate}

\medskip

It is clear that these examples are pointing toward some general motivic interpretation of period integrals. 
Automorphic representations with a nonzero $(H,\chi)$-period are usually characterized in terms of functorial transfers and/or  in terms of some $L$-function attached to $\pi$ having a pole or (not having) a zero at a certain point. In terms of $L$-values, the situation is very similar to a conjecture of Gross on motivic $L$-functions; see \cite[Conjecture 2.7 (ii)]{deligne}. This says that for a critical motive $M$, the order of vanishing of the critical $L$-value $L(\sigma, M, 0)$ is independent of the conjugating automorphism $\sigma.$ In our situation, suppose $\pi$ corresponds to a motive, and suppose having a non-vanishing $(H,\chi)$-period corresponds to the (non-)vanishing of an $L$-value attached to $\pi$ which happens to be a critical $L$-value, then Gross's conjecture would predict the validity of the dictum. For example, the situation in (3), respectively (4), above exactly ties up with critical $L$-values of the underlying Rankin-Selberg $L$-function, 
respectively the standard $L$-function.

\bigskip

\noindent{\it Acknowledgements:} We thank Jeff Adams for several helpful discussions concerning the proof of the results in Section \ref{S:classical}, and Harald Grobner for many comments on
a first draft of this paper which corrected some errors and greatly improved the exposition.
It is also a pleasure to thank Dipendra Prasad, C.S.~Rajan, 
A.~Sankaranarayanan and Jyoti Sengupta for organizing a very memorable international colloquium at TIFR during which this article took shape. This article was completed during the first author's visit at the IHES at Bures-sur-Yvette in July 2012; the first author thanks the IHES for its support and for providing a peaceful yet stimulating working environment.

\vskip 5pt

\section{Cohomological preliminaries} 
\label{sec:cohomology}

We will often talk about a `cohomological cuspidal automorphic representation'. In this section, we will briefly review its definition and discuss some of its very basic properties that we need later. 

Let $G/\Q$ be a connected reductive algebraic group over
$\Q$ and let $S$ be the maximal $\Q$-split
torus in the center $Z$ of $G$. Let $C_{\infty}$ be a maximal compact subgroup of
$G(\R)$ and let $K_\infty = C_\infty S(\R)$. The connected
component of the identity of $K_{\infty}$ is denoted
$K_{\infty}^{\circ}.$ 
For any open-compact subgroup $K_f \subset G(\A_f)$, define the locally symmetric space
$$
S^G_{K_f} := G(\Q)\backslash G(\A)/ K_{\infty}^{\circ} K_f. 
$$
Allowing $K_f$ to vary, one has an inverse system $S^G = \projlim S^G_{K_f}$ of locally symmetric spaces.
\vskip 5pt

Fix a maximal torus $T/\Q \subset G/\Q$, with associated absolute Weyl group $W(G,T) = (N_G(T)/T)(\C)$. Set $X(T) = \Hom_{\C}(T, \mathbb{G}_m)$, so that $W(G,T)$ acts naturally on $X(T)$.
By the theory of highest weight, each $W(G,T)$-orbit $\mu$ in  $X(T)$ corresponds to 
 an  irreducible algebraic representation $E_{\mu}$ of $G(\C)$ with highest weight $\mu$. Let $\E_{\mu}$ be the associated (inverse system of) sheaf on $S^G $. We are interested in the 
sheaf cohomology groups 
$$
H^{\bullet}(S^G, \E_{\mu}): = \injlim_{K_f}  H^{\bullet}(S^G_{K_f} , \E_{\mu})
$$
on which the finite adelic group $G(\A_f)$ acts naturally.
\vskip 5pt

We can compute the above sheaf cohomology via the de~Rham complex, and then reinterpreting the de~Rham complex in terms of the complex computing relative Lie algebra cohomology, we get the isomorphism: 
$$
H^{\bullet}(S^G, \E_{\mu})  \simeq H^{\bullet}(\g_{\infty}, K_{\infty}^{\circ}; \  C^{\infty}(G(\Q)\backslash G(\A)) \otimes E_{\mu}) .
$$
 The inclusion $C^{\infty}_{\rm cusp} (G(\Q)\backslash G(\A)) \hookrightarrow C^{\infty}(G(\Q)\backslash G(\A))$ of the space of smooth cusp forms  in the space of all smooth functions induces, via well-known results of 
Borel \cite{borel-duke},  an injection in cohomology; this defines cuspidal cohomology: 
$$
\xymatrix{
H^{\bullet}(S^G, \E_{\mu}) 
\ar[rr] & &
H^{\bullet}(\g_{\infty},K_{\infty}^{\circ}; C^{\infty}(G(\Q)\backslash G(\A)) \otimes E_{\mu})  \\
H^{\bullet}_{\rm cusp}(S^G, \E_{\mu}) \ar@{^{(}->}[u]
\ar[rr] 
& & 
H^{\bullet}(\g_{\infty},K_{\infty}^{\circ}; C^{\infty}_{\rm cusp}(G(\Q)\backslash G(\A))
\otimes E_{\mu}) \ar@{^{(}->}[u]
}$$
Using the usual decomposition of the space of cusp forms into a direct sum of cuspidal automorphic representations, we get the following fundamental decomposition 
\begin{equation}
H^{\bullet}_{\rm cusp}(S^G, \E_{\mu}) = \bigoplus_{\Pi} 
H^{\bullet}(\g_{\infty},K_{\infty}^{\circ};  \Pi_{\infty} \otimes  E_{\mu}) \otimes \Pi_f
\end{equation}
of $\pi_0(G_{\infty}) \times G(\A_f)$-modules.
\vskip 5pt

We say that {\it $\Pi$ is a cohomological cuspidal automorphic representation} if  
$\Pi$ has a nonzero contribution to the above decomposition for some $\mu$,
or equivalently, if $\Pi$ is a cuspidal automorphic representation whose representation at infinity $\Pi_{\infty},$ after twisting by $E_{\mu},$
has nontrivial relative Lie algebra cohomology. In this situation, we write $\Pi \in {\rm Coh}(G, \mu)$. 
\vskip 5pt

One may also  consider cohomology with compact supports $H^\bullet_c(S^G, \E_{\mu})$. {\it Inner cohomology} is defined as the image of compactly supported cohomology in global cohomology: 
$$
H^\bullet_!(S^G, \E_{\mu}) := {\rm Image}
\left(H^\bullet_c(S^G, \E_{\mu}) \longrightarrow H^\bullet(S^G, \E_{\mu})\right). 
$$
In the literature, inner cohomology is also called {\it interior} or {\it parabolic} cohomology. 
It is a fundamental fact (which comes from analyzing the long-exact sequence arising from the Borel-Serre compactification; see, for example, Li--Schwermer \cite{li-schwermer}) that 
\begin{equation}
\label{eqn:cuspidal-inner}
H^\bullet_{\rm cusp}(S^G, \E_{\mu}) \subset H^\bullet_!(S^G, \E_{\mu}). 
\end{equation}
On the other hand, since any compactly supported function is square-integrable, we also have that inner cohomology sits inside the cohomology group whose elements are represented by square-integrable automorphic forms, i.e., 
\begin{equation}
\label{eqn:inner-l2}
 H^\bullet_!(S^G, \E_{\mu})  \subset 
H^\bullet_{(2)}(S^G, \E_{\mu}). 
\end{equation}
\vskip 5pt

Now let us  briefly recall the action of $\autc$ on the various objects introduced above. First, observe that $\autc$ acts naturally on $X(T)$ by
\[  ({^\sigma}\chi)(t)  = \sigma(\chi(\sigma^{-1}(t)), \quad \text{for $\sigma \in \autc$,  $t \in T(\C)$ and $\chi \in X(T)$.} \]
Similarly, $\autc$ acts naturally  on the set of equivalence classes of  irreducible algebraic representations $(E, \rho)$ of $G(\C)$ (where $\rho: G \rightarrow \GL(E)$) by
\[  {}^\sigma\!\rho(g) =  \sigma (\rho(\sigma^{-1}(g)) \quad \text{for $g \in G(\C)$,} \]
where, the $\autc$-action on $\GL(E) \cong \GL_n(\C)$ is with respect to the canonical Chevalley structure over $\Q$.  
In particular, it follows that if $E = E_{\mu}$ has highest weight $\mu$, then ${}^\sigma\! E_{\mu}$ 
has highest weight ${}^\sigma\! \mu$. 
\vskip 5pt

Another description of the $\sigma$-conjugated algebraic representation $({}^\sigma\!E, {}^\sigma\!\rho)$ is as follows. Set $E_{\sigma} = E \otimes_{\C, \sigma} \C$. Then $({}^\sigma\!E, {}^\sigma\!\rho)$ is isomorphic to the representation of $G(\C)$ on $E_{\sigma}$ with $g \in G(\C)$ acting by
\[   v \mapsto  \rho(\sigma^{-1}(g))(v). \] 
 In particular, when restricted to $G(\Q)$, $^{\sigma}E$ is realized  on $E_{\sigma} = E \otimes_{\C,\sigma} \C$ with $G(\Q)$
acting via its action on the first component $E$ of the tensor product. Thus, there is a natural $\sigma$-linear, $G(\Q)$-equivariant map
\[  E \longrightarrow {}^\sigma\!E. \]
\vskip 5pt

In an analogous way, $\autc$ acts naturally on the smooth representations $(W, \Pi_f)$ of $G(\A_f)$. Namely, we define ${^\sigma}\Pi_f$ to be the action of $G(\A_f)$ on $W \otimes_{\C, \sigma} \C$  
with $G(\A_f)$ acting on the first component $W$ of the tensor product. 
\vskip 5pt

Now, passing to sheaves on the locally symmetric space $S^G$, the above considerations lead to 
a commutative diagram, where the horizontal arrows are  $\sigma$-linear $G(\A_f)$-equivariant isomorphisms
$$
\xymatrix{
H^{\bullet}(S^G, \E_{\mu}) 
\ar[rr] & &
H^{\bullet}(S^G, {}^\sigma\!\E_{\mu})  \\
H^{\bullet}_{!}(S^G, \E_{\mu}) \ar@{^{(}->}[u]
\ar[rr] 
& & 
H^{\bullet}_{!}(S^G, {}^\sigma\!\E_{\mu}) \ar@{^{(}->}[u]
}$$

Thus we have the following 
\vskip 5pt

\begin{prop} \label{prop:generalG}
Suppose that $\pi$ is a cohomological cuspidal automorphic representation of $G(\A)$. Then for any $\sigma \in \autc$, there exists an automorphic representation $\tau_{\sigma}$ appearing in the automorphic discrete spectrum of $G(\A)$ such that:
\begin{itemize}
\item $ \tau_{\sigma,f}= {}^\sigma\pi_f$;
\item $\tau_{\sigma, \infty}$ has nonzero Lie algebra cohomology with respect to ${}^\sigma \mu$.
\end{itemize}
\end{prop}
\begin{proof}
 By assumption, $\pi_f$ occurrs as a $G(\A_f)$-summand in $H^\bullet_{\rm cusp}(S^G, \E_{\mu})$ for some $\E_{\mu}$. We deduce by (\ref{eqn:cuspidal-inner}), (\ref{eqn:inner-l2}) and the above commutative diagram  that 
  ${}^\sigma\pi_f$ occurs in $H^\bullet_!(S^G, {^\sigma}\E_{\mu})$ and hence in $H^\bullet_{(2)}(S^G,  {^\sigma}\E_{\mu})$.
  This proves the proposition.
  \end{proof}
 \vskip 5pt

Let us now suppose that $G$ is a connected reductive group defined over a number field $F$ with ring of adeles $\A_F$ and let $T \subset G$ be a maximal torus defined over $F$. 
We may apply the above discussion to the reductive group  $G_0 :=  {\rm Res}_{F/\Q} G$ over $\Q$ containing the torus $T_0 = {\rm Res}_{F/\Q} T$.  In this case, one has
\[  G_0 \times_{\Q} \C \ \cong \prod_{\tau \in \Hom(F,\C)} G_{\tau} \qquad \text{with} \qquad G_{\tau} := G \times_{F,\tau} \C \]
and
\[  T_0 \times_{\Q} \C \ \cong \prod_{\tau \in \Hom(F,\C)} T_{\tau} \qquad \text{with} \qquad  T_{\tau}:= T \times_{F,\tau} \C, \]
so that
\[  X(T_0) \ \cong \bigoplus_{\tau \in \Hom(F,\C)} X(T_{\tau}), \quad \text{with}\qquad 
X(T_{\tau}) =  \Hom_{\C}( T_{\tau}, \mathbb{G}_m). \]
Note that $X(T_{\tau})$ comes equipped with a natural action of ${\rm Aut}(\C/\tau(F))$. 
Thus, an irreducible algebraic representation $E$ of $G_0(\C)$ is of the form $E = \bigotimes_{\tau} E_{\mu_\tau}$ where $E_{\mu_{\tau}}$ is an irreducible algebraic representation of $G_{\tau}$. 
\vskip 5pt

Let us explicate the action of $\autc$ on $X(T_0)$. For $\sigma \in \autc$ and $\tau \in \Hom(F,\C)$, with $\tau' = \sigma \circ \tau$, the automorphism $\sigma$ induces:
\begin{itemize}
\item[(a)]  a natural isomorphism
\[  {\rm Aut}(\C/\tau(F)) \longrightarrow {\rm Aut}(\C/\tau'(F)) \]
sending $\phi$ to $\sigma \circ \phi \circ \sigma^{-1}$;
\item[(b)]  a natural equivariant isomorphism 
\[  \sigma_*: X(T_{\tau}) \longrightarrow X(T_{\tau'}) \]
given by:
\[  \sigma_*: \chi \mapsto  \sigma \circ  \chi \circ p_{\sigma}, \]
where 
\[ p_{\sigma} :  T_{\tau'} = T_{\tau}  \times_{\C,\sigma} \C \longrightarrow T_{\tau} \]
is the natural projection. Here,  the equivariance of $\sigma_*$ is with respect to the action of ${\rm Aut}(\C/\tau(F))$ on the source  and the action of ${\rm Aut}(\C/\tau'(F))$ on the target, via the isomorphism in (a). 
\end{itemize}
\vskip 10pt

Then the action of $\sigma \in \autc$ on $X(T_0)$ is via the bijections $\sigma_*$ described above. In particular, for any fixed $\tau \in \Hom(F,\C)$, 
\[  X(T_0) \cong {\rm Ind}^{\autc}_{{\rm Aut}(\C/\tau(F))} X(T_{\tau})\]
as $\autc$-modules.
\vskip 10pt
 
 As examples, consider the following cases:
 \vskip 5pt
 
 \begin{itemize}
 \item when $G$ is $F$-split, the action of ${\rm Aut}(\C/\tau(F))$ on $X(T_{\tau})$ is trivial, so that $\sigma_* : X(T_{\tau}) \longrightarrow X(T_{ \tau'})$ is independent of $\sigma$ (subject to $\sigma \circ \tau = \tau'$). In other words, one has a canonical identification $X(T_{\tau}) \leftrightarrow X(T_{\tau'})$. Thus, in this case, the action of $\autc$ on $X(T_0)$ is via the permutation of the components $X(T_{\tau})$, so that if $E = \bigotimes_{\tau} E_{\mu_{\tau}}$, then
 \[  {}^\sigma\!E  = \bigotimes_{\tau} E_{\mu_{\sigma^{-1} \circ \tau}}. \]
\vskip 5pt

\item when $G$ is an inner form of a split group over $F$, the action of ${\rm Aut}(\C/\tau(F))$ on $X(T_{\tau})$ factors through the action of the Weyl group $W_{\tau} := W(G_{\tau}, T_{\tau})$. Thus, one still has a canonical bijection $X(T_{\tau})/ W_{\tau} \leftrightarrow X(T_{\tau'})/ W_{\tau'}$, i.e., between the sets of highest weights. As in the split case, one still has
 \[  {}^\sigma\!E  = \bigotimes_{\tau} E_{\mu_{\sigma^{-1} \circ \tau}}. \]
 \end{itemize}
 \vskip 10pt
 
Finally,  note that the set $S_{\infty}$ of infinite places of $F$ is the set of  orbits of complex conjugation on $\Hom(F,\C)$. If we write $G_{\infty} = {\rm Res}_{F/\Q}(G)(\R) = \prod_{v \in S_{\infty}} G_v$, then $\mu \in X(T_0)$ defines a finite dimensional representation $E = \otimes_v E_{\mu_v}$ via: 
 \begin{itemize}
 \item if $v =\tau \in \Hom(F, \R)$ is real, then $E_{\mu_v} = E_{\mu_{\tau}}$;
 \item if $v = \{ \tau, \overline{\tau} \} \subset \Hom(F,\C)$ is complex, then 
 $E_{\mu_v} =  E_{\mu_{\tau}} \otimes E_{\mu_{\overline{\tau}}}$ as a representation of 
 $G_v  = \{ (g, \overline{g}): g \in G_{\tau} \} \subset G_{\tau}  \times G_{\overline{\tau}}\}$. 
 \end{itemize}
 
 \vskip 5pt

 \section{Cohomological Representations of $\GL_n$} 
 \label{sec:GL(N)}
 In this section, we discuss in greater depth the case when 
 $G = \GL_n/F$ and  recall  a fundamental result due to 
Clozel~\cite[Th\'eor\`eme 3.13]{clozel}, which refines Proposition \ref{prop:generalG}.

\begin{thm}[Clozel] 
\label{thm:clozel}
Let $\Pi$ be a cuspidal automorphic representation of $\GL_n(\A_F)$ such that $\Pi \in {\rm Coh}(\GL_n/F, \mu)$. 
For any $\sigma \in \autc,$ there is a cuspidal automorphic representation 
${}^{\sigma}\Pi \in {\rm Coh}(\GL_n/F, {}^\sigma \mu)$ whose finite part is
${}^\sigma\Pi_f$. 
 \end{thm}
 Thus, the extra information contained here is that the representation $\tau_{\sigma}$ in Proposition \ref{prop:generalG} is cuspidal.
For the precursor to this result of Clozel,    see 
Shimura~\cite[Section 2]{shimura-duke} for  the classical situation of Hilbert modular forms; for representations of $\GL_2$, see Harder~\cite{harder-general} and Waldspurger~\cite{waldspurger}. 
We also remark that in \cite{clozel}, Clozel works with the notion of the ``infinity type $p(\Pi)$"  of a cohomological cuspidal representation $\Pi$, which consists of the exponents appearing in the Langlands parameter of $\Pi_v$ for $v \in S_{\infty}$.  The infinity type $p(\Pi) = \{ p_{\tau}: \tau \in \Hom(F,\C) \}$ is basically the infinitesimal character of  the finite-dimensional representation $E _\mu = \otimes_{\tau \in \Hom(F,\C)}  E_{\mu_{\tau}}$ with  highest weight $\mu$. More precisely, if we assume that $\mu$ is dominant, then for each $\tau \in \Hom(F,\C)$, 
\[  \mu_{\tau} + \rho_n  =   p_{\tau} + ( \frac{n-1}{2},.....,\frac{n-1}{2}), \]
where $\rho_n = (\frac{n-1}{2}, \frac{n-3}{2}, \cdots, -\frac{n-1}{2})$ is the usual half sum of positive roots for $\GL_n$. 
We also refer the reader to \cite[Defn.\,3.6 on p.\,107]{clozel} where the action of $\sigma \in \autc$ on the infinity type  $p(\Pi)$ is defined; it agrees with the action of $\sigma$ on $\mu$ explicated in the previous section.  
\vskip 5pt

Note that in Theorem \ref{thm:clozel}, the archimedean component of ${}^\sigma \Pi$ is not precisely specified: one only knows the finite dimensional representation $E_{{}^\sigma\mu_v}$ with respect to which ${}^\sigma \Pi_v$ has nonzero relative Lie algebra cohomology. 
Thus, one is naturally led to ask: to what extent does  the highest weight $\mu_v$ determine the cohomological representation $\Pi_v$? Let us examine this issue more closely for $\GL_n$.
\vskip 5pt

Assume first that $v = \{ \tau, \overline{\tau} \}$ is a complex  place. Then suppose that 
\[  \mu_{\tau} + \rho_n  = (a_1, a_2,...,a_n) \quad \text{and} \quad \mu_{\overline{\tau}}+\rho_n   = (b_1, b_2,..., b_n), \]
with $a_i$ and $b_i$ decreasing.
By the purity lemma of Clozel \cite[Lemma 4.9]{clozel}, there is a ${\sf w} \in \Z$ (independent of $v$) such that $a_i + b_{n+1-i} = {\sf w}$ for all $i$.  Then 
\[  \Pi_v = {\rm Ind}_B^{\GL_n(\C)} z^{a_1} \overline{z}^{b_n} \times \cdots \times z^{a_n} \overline{z}^{b_1}. \]
In other words, $\Pi_v$ is completely determined by $\mu_v = \{ \mu_{\tau}, \mu_{\overline{\tau}} \}$ if $v$ is complex. 

\vskip 5pt
 Assume now that $v$ is a real place. Then
$\mu_v = (\mu_{v,1}, \dots, \mu_{v,n})$ where $\mu_{v,j} \in \Z$ and $\mu_{v,1} \geq \mu_{v,2} \geq \cdots \geq 
\mu_{v,n}.$  Further, one knows that the highest weight $\mu$ is  pure (because only pure weights support nonzero cuspidal cohomology); namely  there exists a ${\sf w} \in \Z$ (independent of $v$) such that 
\[ \mu_{v,j} + \mu_{v,n-j+1} = {\sf w}. \] 
Note that if $n$ is odd, then  ${\sf w} = 2\mu_{v,(n-1)/2}$ is even. Now put $\ell = (\ell_v)_{v \in S_\infty}$ where 
\[  \ell_v = 2\mu_v - {\sf w} + 2\rho_{n} \]
 Then 
 \[  \ell_{v,j} \ = \ 2 \mu_{v,j} - {\sf w} + n-2j+1 \ = \ \mu_{v,j} - \mu_{v, n-j+1} + n-2j+1. \]
Observe that: 
\[ \begin{cases}
  \ell_{v,1} > \dots > \ell_{v, [n/2]} > 0; \\
  \ell_{v,n-j+1} = -\ell_{v,j}; \\
   \ell_{v,j} \equiv {\sf w} + n -1 \pmod{2}. 
   \end{cases} \]
In particular, $\ell_{v,j} \equiv 0 \pmod{2}$ if $n$ is odd. 
\vskip 5pt

We define an irreducible representation $J_{\mu_v}$ as 
the representation induced from the $(2,\dots,2)$-parabolic if $n=2m$ is even: 
\begin{equation}
\label{eqn:j-mu-even}
J_{\mu_v} = D_{\ell_{v,1}}|\ |^{-{\sf w}/2} \times \dots \times D_{\ell_{v,m}}|\ |^{-{\sf w}/2}, 
\end{equation}
and if $n = 2m+1$ is odd then it is induced from the $(2,\dots,2,1)$-parabolic subgroup: 
\begin{equation}
\label{eqn:j-mu-odd}
J_{\mu_v} = D_{\ell_{v,1}}|\ |^{-{\sf w}/2} \times \dots \times D_{\ell_{v,m}}|\ |^{-{\sf w}/2} \times |\ |^{-{\sf w}/2}, 
\end{equation}
where $D_l$ is the `discrete series' representation of $\GL_2(\R)$ of lowest weight $l+1$ and 
central character ${\rm sgn}^{l+1}.$ Given $\Pi \in {\rm Coh}(\GL_n/F, \mu)$, we know, when $n$ is even, that 
\begin{equation}
\label{eqn:pi-infinity-even}
\Pi_v \ \simeq \ J_{\mu_v}, 
\end{equation}
and when $n$ is odd, we know that
\begin{equation}
\label{eqn:pi-infinity-odd}
\Pi_v \ \simeq \ J_{\mu_v} \otimes {\rm sgn}^{\epsilon(\Pi_v)}, 
\end{equation}
where $\epsilon(\Pi_v) \in \{0,1\}$ is defined by 
$$
(-1)^{\epsilon(\Pi_v)} = (-1)^{(n-1)/2}\omega_{\Pi_v}(-1).
$$
(See, for example, \cite[Section 5.1]{raghuram-shahidi}.) 
Thus, when $v$ is real,  $\mu_v$ completely determines $\Pi_v$ when $n$ is even; however, when $n$ is odd, we need not only 
$\mu_v$ but also the parity of the central character at the real place $v$ to pin down $\Pi_v.$ 
\vskip 5pt

Now we bring in the $\autc$-action. As we have noted in the previous section, for $\sigma \in \autc$, we have ${}^\sigma\!\mu_{\tau}= \mu_{\sigma^{-1} \tau}$ for $\tau \in \Hom(F,\C)$. The above discussion implies that when $n$ is even or when $v$ is complex, the local component ${}^\sigma \Pi_v$  is completely determined by ${}^\sigma \mu_v$. We explicate the situation in two cases:
\vskip 5pt

\begin{prop}  \label{prop:arch}
(i)  Assume first that $F$ is a totally complex quadratic extension of a totally real $F^+$. 
Then for any $\sigma \in \autc$,
\[  {}^\sigma \Pi_v = \Pi_{\sigma^{-1} v}. \]
\vskip 5pt

(ii)  Assume that $F$ is totally real.
When $n$ is even, we have:
\begin{equation}
\label{eqn:n-even-aut-c}
{}^\sigma\Pi_\infty \ = \ \bigotimes_{v \in S_\infty} \Pi_{\sigma^{-1} v} \ = \ \bigotimes_{v \in S_\infty} J_{\mu_{\sigma^{-1}v}}.
\end{equation}
When $n$ is odd, we have:
\begin{equation}
{}^\sigma\Pi_\infty \ = \ \bigotimes_{v \in S_\infty}  \left(J_{\mu_{\sigma^{-1}v}} \otimes {\rm sgn}^{\epsilon(\Pi_v)}\right),
\end{equation}
where $\epsilon(\Pi_v)$ is as defined in (\ref{eqn:pi-infinity-odd}). In particular, if the sign $\omega_{\Pi_v}(-1)$ is independent of the infinite place $v$, then 
\[ {}^\sigma\Pi_\infty \ = \ \bigotimes_{v \in S_\infty} \Pi_{\sigma^{-1} v} \ = \ \bigotimes_{v \in S_\infty} J_{\mu_{\sigma^{-1}v}} \]
as in the case when $n$ is even.
\end{prop}

\begin{proof}
(i)  For each place $v^+$ of $F^+$, let $v = \{\tau, \overline{\tau} \}$ be the place of $F$ over $v^+$, so that  $\tau$ and $\overline{\tau}$ are the two elements of $\Hom(F,\C)$ whose restriction to $F^+$ is $v^+$. Then for any $\sigma \in \autc$, $\sigma^{-1} \circ v := \{ \sigma^{-1} \circ \tau, \sigma^{-1} \circ \overline{\tau} \}$ are the two elements which restrict to $\sigma^{-1} \circ v^+$. Thus, ${}^\sigma \mu_v = \mu_{\sigma^{-1} \cdot v}$ and so we have:
\[  {}^\sigma \Pi_v = \Pi_{\sigma^{-1} \circ v}. \]
 
 \vskip 5pt
 
(ii)  Now assume that $F$ is totally real, so that $S_{\infty}  = \Hom(F,\C)$. The case when $n$ is even follows from our discussion above.
 When $n$ is odd, the situation is a little more tricky and we need to consider central characters. Let $\omega_\Pi$ be the global central character of $\Pi.$ It is of the form: 
\[ \omega_\Pi = \omega_\Pi^\circ \otimes |\ |^{-n{\sf w}/2} \]
 with  $\omega_\Pi^\circ$ a Hecke character of finite order. Let us simplify notations and write this as: 
$\omega = \omega^\circ |\ |^{\sf m}$ where $\omega^\circ$ is a finite-order Hecke character and ${\sf m} \in \Z.$ Then, 
for any $\sigma \in \autc$, we have:
\begin{equation}
\label{eqn:sigma-omega}
({}^\sigma\omega)_v \ = \ \omega_v, \ \  \forall v \in S_\infty. 
\end{equation}
This follows from the following two observations:
\begin{itemize}
\item $({}^\sigma \omega^{\circ}) = \sigma \circ \omega^{\circ}$; this is because the latter is still a continuous character of $F^{\times} \backslash \A_F^{\times}$ as $\omega^{\circ}$ takes value in a finite group. In particular, for real $v$, $\sigma \circ \omega^{\circ}_v = \omega^{\circ}_v$ since $\omega^{\circ}_v$ takes value in $\{ \pm 1\}$.

\item ${}\sigma | \ | = |\ |$; this is because at all finite places $w$, $| \ |_w$ takes value in $\Q$ and so ${}\sigma|\ |_w = |\ |_w$ (and a Hecke character is determined by almost all its local components, by weak approximation). 
\end{itemize}

\vskip 5pt

Next, the global central character of ${}^\sigma\Pi$ satisfies: 
$$
\omega_{{}^\sigma\Pi} = {}^\sigma\omega_\Pi, 
$$ 
which can be seen by checking equality of local characters at all finite unramified places. Hence the parity that is needed in pinning down the representations at infinity as in (\ref{eqn:pi-infinity-odd}) is given by
$\epsilon({}^\sigma\Pi_v) = \epsilon(\Pi_v),$
since 
\[  (-1)^{\epsilon({}^\sigma\Pi_v)} = (-1)^{(n-1)/2}\omega_{{}^\sigma\Pi_v}(-1) \quad \text{and} \quad 
\omega_{{}^\sigma\Pi_v}(-1) = {}^\sigma\omega_{\Pi_v}(-1) = \omega_{\Pi_v}(-1) \]
 by  (\ref{eqn:sigma-omega}). This proves the proposition.
 \end{proof}

\begin{rem}{\rm 
When $F$ is totally real and $n$ is odd, the hypothesis in the above theorem that the sign  
$\omega_{\Pi_v}(-1)$ is independent of $v$ is arithmetically interesting because it is a necessary condition for 
the standard $L$-function of $\Pi$ to have a critical point. This will also be the case if the rank $n$ Grothendieck motive $M = M_\Pi$ over $F$ that is conjecturally attached to $\Pi$ is {\it special}, i.e., has the property that complex conjugation acts via the `same' scalar on the middle Hodge type for every real embedding $v$ of $F$; see, for example, Blasius~\cite[M3]{blasius}. 
}\end{rem}

\vskip 10pt

\section{The $\GL_2$-examples}
\label{sec:gl2}
After the preliminary discussions of the previous two sections, we can now begin the consideration of periods.  In this section, we illustrate the question we will study for the case of $\GL_2$. Let us, for the sake of simplicity, take $G = \GL_2/\Q$, although everything discussed in this section works for $\GL_2$ over any number field. 

\subsection{Whittaker periods}
\label{sec:gl2-whitakker}

For the subgroup $H$ we take 
$$
H = U_2 = U = \left\{ \left(\begin{array}{ll} 1 & x \\ 0 & 1 \end{array}\right) \ : \ x \in {\mathbb G}_a \right\}, 
$$
i.e., $U$ is the unipotent radical of the standard Borel subgroup of upper triangular matrices in $G.$ 
Fix a nontrivial additive character $\psi : \Q\backslash \A \to \C^\times$. Then, as usual, $\psi$ gives a character 
$\psi : U(\Q)\bs U(\A) \to \C^\times$ by $\psi \left(\begin{smallmatrix}1 & x \\ 0 & 1\end{smallmatrix}\right) = \psi(x).$ 
Using the same symbol $\psi$ for both the characters will cause no confusion. 
In this situation, the linear functional $\ell_\psi$ defined in (\ref{eqn:l-chi}) is called a global Whittaker functional. 

Given a cuspidal automorphic representation $(\pi,V_\pi)$ of $\GL_2(\A)$ we can define for each $\phi \in V_\pi$ the associated Whittaker vector 
\begin{equation}
\label{eqn:w-phi}
W_\phi(g) \ := \ \int_{U(\Q)\bs U(\A)} \phi(ug) \psi(u)^{-1}\, du. 
\end{equation}
Observe that $W_\phi(1) = \ell_\psi(\phi)$. Using the action of $\GL_2(\A)$ we see that $\ell_\psi(\phi) \neq 0$ for some 
$\phi$ if and only if $W_\phi \neq 0$ for some $\phi$. 
A fundamental fact at the heart of the $\GL_2$-theory of automorphic forms is
that $W_\phi$ determines $\phi$. (See, for example, \cite[Lecture 4, Section 1]{cogdell}.) Indeed,  we have a Fourier expansion of the form
\begin{equation}
\label{eqn:fourier}
\phi(g) \ = \ 
\sum_{\gamma \in \Q^\times}  
W_\phi \left(\left(\begin{array}{ll} \gamma & 0 \\ 0 & 1 \end{array}\right) g\right), 
\end{equation}
In particular, every cuspidal automorphic representation has a nonvanishing Whittaker period. 

Now let us suppose that $\pi$ is a cohomological cuspidal automorphic representation of $\GL_2(\A)$, and in particular, 
$\pi$ has nonvanishing Whittaker periods. For any 
$\sigma \in \autc$ Theorem~\ref{thm:clozel} says that ${}^\sigma\pi$ is also a (cohomological) cuspidal automorphic representation of $\GL_2(\A)$. Hence, by the above discussion once again, ${}^\sigma\pi$ also has nonvanishing Whittaker periods, i.e, we have arithmeticity for Whittaker periods for $\GL_2.$ 

The main ingredients in arithmeticity for Whittaker periods are (\ref{eqn:fourier}) and Theorem~\ref{thm:clozel}. Both these ingredients, which are nontrivial assertions, are valid for $\GL_n/F$ over any number field $F$ after suitable modification;  for example, the Fourier expansion takes the form: 
\begin{equation}
\label{eqn:fourier-gln}
\phi(g) \ = \ 
\sum_{\gamma \in \GL_{n-1}(F)/U_{n-1}(F)}  
W_\phi \left(\left(\begin{array}{ll} \gamma & 0 \\ 0 & 1 \end{array}\right) g\right), 
\end{equation}
(See, for example, \cite[loc.\,cit.]{cogdell}.) Hence we get arithmeticity for Whittaker periods for $\GL_n/F$. In Section~\ref{sec:whittaker-classical} we consider the case of classical groups, especially split ${\rm SO}(2n+1)$, where the analysis is far more complicated. 

\medskip

Reverting to $\GL_2/\Q$, let us go through the analysis for arithmeticity for Whittaker periods in the classical context of modular forms. Fix a positive  integer $N$ and consider the space $S_k(N)$ consisting of all 
holomoprhic cusp forms of weight $k$ on the upper half plane for the discrete subgroup 
$\Gamma_1(N)$ of $\SL_2(\R).$ A cusp form $\varphi \in S_k(N)$ has a Fourier expansion 
$$
\varphi(z) = \sum_{n=1}^\infty a_n(\varphi) e^{2 \pi inz}.
$$
Now define $S_k(N,\Q)$ to be the $\Q$-subspace of the $\C$-vector space $S_k(N)$ consisting of all $\varphi$ such that 
$a_n(\varphi) \in \Q$ for all $n \geq 1.$ One has the following {\it nontrivial} fact: 
\begin{equation}
\label{eqn:skn}
S_k(N) \ = \ S_k(N,\Q) \otimes_\Q \C.
\end{equation}
(See, for example, Shimura \cite[Theorem 3.52]{shimura-book}.) This may be stated as the fact that the space of cusp forms of weight $k$ and level $N$ has a basis of cusp forms all of whose Fourier coefficients are in $\Q.$ Indeed, there is a deeper integrality statement which says that the above is true with $\Z$ instead of $\Q$; however, for our purposes, a $\Q$-basis is sufficient. 
Let us note that (\ref{eqn:skn}) is the classical analogue of the statement (see Clozel~\cite[Th\'eor\`eme 3.19]{clozel}) that cuspidal cohomology for $\GL_n/F$ admits a suitable rational structure. 
Now, given $\varphi \in S_k(N)$ and $\sigma \in \autc$ we can define a function ${}^\sigma\varphi$ via a $q$-expansion. 
$$
{}^\sigma\varphi(z) \ := \ \sum_{n=1}^\infty  \sigma(a_n(\varphi)) e^{2 \pi inz}. 
$$
It follows from (\ref{eqn:skn}) that 
${}^\sigma\varphi \in S_k(N).$ This is the classical analogue of Theorem~\ref{thm:clozel}. Arithmeticity for Whittaker models takes the form: 
$$
a_n(\varphi) \neq 0 \  \Longrightarrow \ a_n({}^\sigma\varphi) \neq 0, 
$$
which is built into the definition of the Galois conjugate ${}^\sigma\varphi.$ The depth of the phenomenon lies in the rationality statement in (\ref{eqn:skn}).

\subsection{$\GL_1$-periods}
\label{sec:gl2-gl1}

We continue with $G = \GL_2/\Q$ and now we take 
$$
H = \left\{ \left(\begin{array}{ll} x & 0 \\ 0 & 1 \end{array}\right) \ : \ x \in {\mathbb G}_m \right\} \simeq \GL_1 . 
$$
Take a Hecke character $\chi : \Q^\times\bs\A^\times \to \C^\times$, which gives a character 
$\chi : H(\Q) \bs H(\A) \to \C^\times$ by 
$\chi\left(\begin{smallmatrix} x & 0 \\ 0 & 1 \end{smallmatrix} \right) = \chi(x).$ 
Using the same symbol $\chi$ for both the characters will cause no confusion.
Consider a cuspidal automorphic representation $\pi$ of $\GL_2(\A)$. Suppose there is a $\phi \in V_\pi$ such that 
$$
\ell_\chi(\phi)  \ = \ 
\int_{x \in  \Q^\times\bs\A^\times} \phi \left(\left(\begin{array}{ll} x & 0 \\ 0 & 1 \end{array}\right)\right) 
\chi(x) \, dx \ 
\neq \ 0. 
$$
To  analyze these integrals, and to relate them to $L$-values, following Jacquet-Langlands \cite{jacquet-langlands}, fix a nontrivial additive character $\psi$ as in the previous subsection and 
consider the Whittaker model $\W(\pi) = \W(\pi,\psi)$ of $\pi.$ Let 
the cusp form $\phi$ correspond to $W_\phi \in \W(\pi)$ where $W_\phi$ is defined in (\ref{eqn:w-phi}). Then for a complex variable $s$ such that $\Re(s) \gg 0,$ the classical unfolding argument gives:
\begin{eqnarray*}
\ell(s, \phi, \chi) & := & \int_{x \in  \Q^\times\bs\A^\times} \phi \left(\left(\begin{array}{ll} x & 0 \\ 0 & 1 \end{array}\right)\right) 
\chi(x) |x|^{s-\tfrac12} \, dx \\
& = &  
\int_{x \in  \A^\times} W_\phi \left(\left(\begin{array}{ll} x & 0 \\ 0 & 1 \end{array}\right)\right) 
\chi(x) |x|^{s-\tfrac12} \, dx.
\end{eqnarray*}
Denote the {\it zeta integral} on the right hand side by $Z(s, W_\phi, \chi), $ and note that all the ingredients in that integral are  factorizable. Changing notation if necessary, there is a cusp form $\phi$ so that $\ell_\chi(\phi) \neq 0$ and the associated Whittaker vector $W_\phi$ is a pure-tensor 
$W_\phi = \otimes W_p.$ Outside a finite set of primes $S$ containing the infinite prime and all the primes where $\pi$ or $\chi$ is ramified, one knows that $W_p$ is the spherical vector normalized so that $W_p(1) = 1$, and the local zeta-integral computes the local $L$-function: 
$$
Z(s, W_p, \chi_p) \ = \ \int_{x_p \in \Q_p^\times} W_p \left(\left(\begin{array}{ll} x_p & 0 \\ 0 & 1 \end{array}\right)\right) 
\chi_p(x_p) |x_p|_p^{s-\tfrac12} \, dx_p \ = \ L(s, \pi_p \otimes \chi_p). 
$$
(See, for example, Gelbart~\cite[Prop.\ 6.17]{gelbart}.)
Let $L^S(s,\pi \otimes \chi) := \prod_{p \notin S} L_p(s, \pi_p \otimes \chi_p)$ denote the partial $L$-function. So far we have: 
$$
\ell(s,\phi,\chi) = Z(s,W_\phi, \chi) = \left(\prod_{p \in S}Z(s,W_p, \chi_p) \right) \cdot L^S(s, \pi \otimes \chi).
$$
Now multiply and divide the right hand side by the local $L$-factors at $p \in S$ to get: 
\begin{equation}
\label{eqn:gl2-zeta}
\ell(s,\phi,\chi) = \left(\prod_{p \in S} \frac{Z(s,W_p, \chi_p)}{L_p(s,\pi_p \otimes \chi_p)} \right) \cdot L(s, \pi \otimes \chi).
\end{equation}
The integral $\ell(s,\phi,\chi)$ converges for all $s$ since $\phi$ is rapidly decreasing. On the right hand side, one knows from Jacquet-Langlands that each of the ratios $Z(s,W_p, \chi_p)/L_p(s,\pi_p \otimes \chi_p)$, {\it a priori} defined only for $\Re(s) \gg 0$, in fact have an analytic continuation to all of $s$ (see, for example, \cite[Theorem 6.12 (ii)]{gelbart}), 
and that the completed $L$-function $L(s, \pi \otimes \chi)$ is an entire function of $s$ 
(see, for example, \cite[Theorem 6.18]{gelbart}). 
We can now prove the following characterization of the existence of $(\GL_1,\chi)$-periods and 
nonvanishing of a certain $L$-value: 

\begin{prop}
\label{prop:gl2-characterize}
Let $\pi$ be a cuspidal automorphic representation of $\GL_2(\A)$, and $\chi$ a Hecke character of $\Q$. Then, the following are equivalent: 
\begin{enumerate}
\item There exists a cusp form $\phi \in V_\pi$ such that $\ell_\chi(\phi) \neq 0.$
\item $L(\tfrac12, \pi \otimes \chi) \neq 0.$
\end{enumerate}
\end{prop}

\begin{proof}
For (1) $\Rightarrow$ (2), put $s = 1/2$ in (\ref{eqn:gl2-zeta}) to get:
$$
0 \neq \ell_\chi(\phi) = r \cdot L(\tfrac12, \pi \otimes \chi)
$$
where $r$ is an ad-hoc notation for the product $\prod_{p \in S} Z(\tfrac12,W_p, \chi_p)/L_p(\tfrac12,\pi_p \otimes \chi_p).$ Hence the right hand side is not zero. 

For  (2) $\Rightarrow$ (1), given $L(\tfrac12, \pi \otimes \chi) \neq 0,$ to construct a cusp form $\phi$ with non-vanishing period, we construct the associated Whittaker vector $W_\phi$ as a pure-tensor. Outside a finite set $S$ as above, take $W_p$ to be the normalized spherical vector. For places in $S$, given any $s_0$ (for us $s_0=1/2$), 
we are always guaranteed the existence of a Whittaker vector $W_p$ such that the ratio 
$Z_p(s_0, W_p, \chi_p)/L(s_0, \pi_p \otimes \chi_p) \neq 0.$ See \cite[(6.29)]{gelbart}. ({\bf Note:} Indeed, for $\GL_2$ there is a $W_p$ for each place $p$ so that the local zeta integral computes the local $L$-factor, and so the ratio is in fact $1$.  
We deliberately stated it in a weaker form of just nonvanishing of that ratio as that is the way it will generalize to 
$\GL_n \times \GL_{n-1}$.) Now put all the local Whittaker vectors to get a global Whittaker vector, and take $\phi$ to be the associated cusp form. The proof follows again from (\ref{eqn:gl2-zeta}) at $s=1/2.$ 
\end{proof}

\begin{rem}{\rm Observe that it is possible for $L(\tfrac12, \pi \otimes \chi) \neq 0$ and yet $L_f(\tfrac12, \pi \otimes \chi) = 0.$
(We use $L(s,\dots)$ for the completed $L$-function, and $L_f(s,\dots)$ for the finite part.) Such a phenomenon will happen when the infinite part $L_\infty(s, \pi \otimes \chi)$ has a pole at $s = 1/2.$ Here is an easy example: 
Let $\Delta \in S_{12}(\SL_2(\Z))$ be the Ramanujan $\Delta$-function which is a weight $12$ cusp form of full level. 
Let $\pi := \pi(\Delta) \otimes |\ |^{-6}$ and take $\chi$ to be the trivial character. 
(For us, cuspidal automorphic representations need not be unitary, and indeed, $\pi$ is not unitary.) Anyway, let $L(s,\pi)$ be the Jacquet-Langlands $L$-function, and $L(s,\Delta)$ be the classical Hecke $L$-function; then $L(s, \pi) = L(s-6, \pi(\Delta)) = L(s-1/2, \Delta).$ Using the classical functional equation $L(s, \Delta) = L(12-s, \Delta)$ we get
$$
L(\tfrac12, \pi)  \ = \ L(0,\Delta) \ = \ L(12, \Delta)  \ \neq \ 0, 
$$
The $L$-factor at infinity is given by: 
$$
L_\infty(s, \pi_\infty) \ = \ L_\infty(s -6, \pi(\Delta)_\infty) \ = \  2 \,(2\bfgreek{pi})^{-s+\tfrac12}\, \Gamma\left(s - \tfrac12\right)
$$
(For the last equation, see, for example, \cite[4.4]{raghuram-tanabe}; the presence of an additional factor of $2$ makes no difference to the discussion at hand.) Hence, $L_\infty(s, \pi_\infty)$ has a pole at 
$s = 1/2,$ in other words, {\it nonvanishing of the global $L$-function at a (seemingly interesting) point  does not guarantee that the point is a critical point.} 
}\end{rem}

\smallskip

Now, given a cuspidal representation $\pi$ as above with a nonvanishing $(\GL_1, \chi)$-period, and consequently with $L(\tfrac12, \pi \otimes \chi) \neq 0$, we want to analyze the dictum of arithmeticity, whence we take $\pi$ to be of cohomological type. But, even if $\pi$ is of cohomological type with $L(\tfrac12, \pi \otimes \chi) \neq 0,$ it is not guaranteed that $s=1/2$ is a critical point. The same counterexample as in the above remark will work for this. Henceforth, we assume in addition that $s=1/2$ is a critical point; i.e., by definition, $L_\infty(s, \pi\otimes\chi)$ and 
$L_\infty(s, \pi^{\sf v} \otimes\chi^{-1})$ are regular at $s=1/2.$ Since local $L$-values are always nonzero, 
under the additional assumption of criticality of $s = 1/2$, we get
\begin{equation}
\label{eqn:gl2-global-finite}
L(\tfrac12, \pi \otimes \chi) \neq 0 \ \Longleftrightarrow \ L_f(\tfrac12, \pi \otimes \chi) \neq 0.
\end{equation}
Now one can prove arithmeticity, for which we need the following algebraicity theorem due to Manin \cite{manin} in certain special cases, and more generally due to  
Shimura \cite{shimura-math-ann}; for the version stated below, see \cite{raghuram-tanabe}.  

\begin{prop}
\label{prop:manin-shimura}
Let $\pi$ be a cohomological cuspidal automorphic representation of $\GL_2(\A)$. There exists two nonzero complex number $p^\pm(\pi)$ such that if $s =  \tfrac12$ is critical then for any algebraic Hecke character $\chi$, and any 
$\sigma \in \autc$ we have 
$$
\sigma\left(\frac{L_f(\tfrac12, \pi \otimes \chi)}{p^{\epsilon_\chi}(\pi) \G(\chi) (2 \bfgreek{pi} i)^{d_\infty}}\right) \ = \ 
\frac{L_f(\tfrac12, {}^\sigma\pi \otimes {}^\sigma\chi)}{p^{\epsilon_\chi}(\pi) \G({}^\sigma\chi) (2 \bfgreek{pi} i)^{d_\infty}},
$$
where $\G(\chi)$ is the Gau\ss~sum attached to $\chi,$ $\epsilon_\chi$ is a sign keeping track of the parity of $\chi,$ and $d_\infty$ is an integer determined entirely by $\pi_\infty.$ (For more details see \cite{raghuram-tanabe}.)
\end{prop}

A trivial corollary to the above deep proposition is that 
\begin{equation}
\label{eqn:gl2-gross}
L_f(\tfrac12, \pi \otimes \chi) \neq 0 \  \Longleftrightarrow \ 
L_f(\tfrac12, {}^\sigma\pi \otimes {}^\sigma\chi) \neq 0.
\end{equation}
The reader should compare this with Gross's conjecture mentioned in the introduction. 

\begin{thm}[Arithmeticity for $(\GL_1,\chi)$-periods for $\GL_2$]
Let $\pi$ be a cohomological cuspidal automorphic representation of $\GL_2(\A_\Q)$. Suppose that $\pi$ 
has a nonvanishing $(\GL_1, \chi)$-period for an algebraic Hecke character of $\Q$. Suppose, further, that 
$s = 1/2$ is a critical point for the $L$-function $L(s, \pi \otimes \chi)$. Then ${}^\sigma\pi$ has a nonvanishing 
$(\GL_1, {}^\sigma\chi)$-period. 
\end{thm}

\begin{proof}
Follows from Proposition~\ref{prop:gl2-characterize}, (\ref{eqn:gl2-global-finite}) and (\ref{eqn:gl2-gross}). 
\end{proof}

Before closing this section, let us note that the above discussion is equivalent to taking: 
$$
G = \GL_2 \times \GL_1, \ \ {\rm and} \ \ H = \Delta\GL_1 := 
 \left\{ \left(x, \left(\begin{array}{ll} x & 0 \\ 0 & 1 \end{array}\right)\right) \ : \ x \in \GL_1 \right\}.
$$
It is from this perspective that it generalizes readily to the context of $\GL_n$ and $\GL_{n-1}$ which we discuss in 
Section~\ref{sec:gln-gl(n-1)}.

\section{Arithmeticity of Shalika models for $\GL_{2n}$} 

In the remainder of the paper, we shall consider various generalizations of the results in the previous section. One generalization of the Whittaker model for $\GL_2$ to $\GL_n$ is the so-called Shalika model. We will first define the notion of a {\it Shalika model} of a cuspidal automorphic representation $\Pi$ of ${\rm GL}_{2n}(\A)$ where $\A = \A_F$ is the adele ring of a number field $F$;
this particular situation was our original motivation to consider arithmeticity questions for periods.   Let
$$ \S:=\left\{
s =  \left( \!\!\begin{array}{ccc}
h &  0 \\
0 &  h
\end{array}\!\!\right)
\left( \!\!\begin{array}{ccc}
1 &  X\\
0 &  1
\end{array}\!\!\right) \Bigg|
\begin{array}{l}
h\in \GL_n\\
X\in {\rm M}_n
\end{array}\right\}\subset \GL_{2n} =: G.
$$
It is traditional to call $\S$ the Shalika subgroup of $G$. A character $\eta : F^\times \backslash \A^\times \to \C^\times$ and a character $\psi : F \backslash \A \to \C^\times$ can be extended to a character of $\S(\A)$:
$$s=\left(\!\! \begin{array}{ccc}
h &  0\\
0 &  h
\end{array}\!\!\right)\left(\!\! \begin{array}{ccc}
1 &  X\\
0 &  1
\end{array}\!\!\right) \mapsto (\eta \otimes \psi)(s) := \eta(\det(h))\psi(Tr(X)).$$
We will also denote $\eta(s) = \eta(\det(h))$ and $\psi(s) = \psi(Tr(X))$. For a cusp form $\varphi\in\Pi$, and a character 
$\eta$ with  $\eta^n = \omega_\Pi,$ consider the integral
$$
\S^\eta_\psi(\varphi)(g):=\int_{Z_{G}(\A)\S(F)\backslash \S(\A)} (\Pi(g)\cdot\varphi)(s)\eta^{-1}(s)\psi^{-1}(s) ds, \ \ 
g\in \GL_{2n}(\A). 
$$
When $n=1$, observe that $S^{\eta}_{\psi}$ is simply the $\psi$-Whittaker period of $\GL(2)$, since $\eta$ is forced to be the central character of $\Pi$.
\vskip 5pt

The following theorem, due to the works of many people (Jacquet--Shalika \cite{jacquet-shalika}, 
Asgari--Shahidi \cite{asgari-shahidi-generic, asgari-shahidi}, Hundley--Sayag \cite{hundley-sayag}) gives a necessary and sufficient condition for $\S^\eta_\psi$ to be non--zero.

\begin{thm}
\label{thm:shalika} 
Let $\Pi$ be a cuspidal automorphic representation of $\GL_{2n}(\A_F)$. 
For a pair of characters $(\eta, \psi)$, the following are equivalent:
\begin{enumerate}
\item[(i)] There is a $\varphi\in\Pi$ and $g\in G(\A)$ such that $\S^\eta_\psi(\varphi)(g)\neq 0$.
\item[(ii)] $\S^\eta_\psi$ defines an injection of $G(\A)$-modules
$$
\Pi\hookrightarrow \textrm{\emph{Ind}}_{\S(\A)}^{G(\A)}[\eta\otimes\psi].
$$
\item[(iii)] Let $S$ be any finite set of places containing $S_{\Pi,\eta}$.
The twisted partial exterior square $L$-function
$$
L^S(s,\Pi,\wedge^2\otimes\eta^{-1}):=\prod_{v\notin S} L(s,\Pi_v,\wedge^2\otimes\eta^{-1}_v)
$$
has a pole at $s=1$.
\item[(iv)] $\Pi$ is the transfer of a globally generic cuspidal automorphic representation $\pi$ of 
${\rm GSpin}_{2n+1}(\A)$ whose central character $\omega_\pi = \eta.$

\end{enumerate}
Moreover, when these conditions hold, the transfer in (iv) is strong at all archimedean places, in the sense that it respects L-parameters.
\end{thm}

If $\Pi$ satisfies any one, and hence all, of the equivalent conditions of Theorem \ref{thm:shalika}, then we say that $\Pi$ {\it has an $(\eta,\psi)$-Shalika model}, and we call the isomorphic image $\S^\eta_\psi(\Pi)$ of $\Pi$ under $\S^\eta_\psi$ a {\it global $(\eta,\psi)$-Shalika model} of $\Pi$.

\vskip 5pt

There is a companion theorem to Theorem \ref{thm:shalika} which considers the ${\rm Sym}^2$ L-function (see \cite{asgari-shahidi-generic, asgari-shahidi} and \cite{hundley-sayag}):
\vskip 5pt

\begin{thm} \label{thm:Sym2}
Let $\Pi$ be a cuspidal automorphic representation of $\GL_{2n}(\A_F)$. 
 Then the following are equivalent:
\begin{enumerate}
 \item[(i)] Let $S$ be any finite set of places containing $S_{\Pi,\eta}$.
The twisted partial symmetric square $L$-function
$$
L^S(s,\Pi,{\rm Sym}^2\otimes\eta^{-1}):=\prod_{v\notin S} L(s,\Pi_v,{\rm Sym}^2\otimes\eta^{-1}_v)
$$
has a pole at $s=1$.
\item[(ii)] $\Pi$ is the transfer of a globally generic cuspidal automorphic representation $\pi$ of 
${\rm GSpin}_{2n}(\A)$ with  connected central character $\omega^0_\pi = \eta.$
\end{enumerate}
Moreover, when these conditions hold, the transfer in (ii) is strong at all archimedean places, in the sense that it respects L-parameters.
\end{thm}

\vskip 5pt

Finally, here is the main result of this section:
\vskip 5pt

\begin{thm}(Arithmeticity of Shalika periods.) 
\label{thm:arithmeticity-shalika}
Suppose that $F$ is totally real. Let $\Pi$ be a cohomological cuspidal automorphic representation of 
$\GL_{2n}(\A_F)$ which has an $(\eta, \psi)$-Shalika model.  Then, for any $\sigma \in {\rm Aut}(\C)$, ${}^{\sigma}\Pi$ is a cohomological cuspidal automorphic representation with a 
$({}^{\sigma}\eta, \psi)$-Shalika model. 
\end{thm}

\begin{proof} The proof is the content of the appendix of \cite{grobner-raghuram-2}; but for the sake of completeness we give a brief sketch here, elaborating upon certain important points. 
  The hypothesis that $\Pi$ is a cohomological cuspidal representation having an $(\eta, \psi)$-Shalika model imposes certain restrictions on $\eta.$ In particular, we claim that 
{\it $\eta_v$ is independent of $v \in S_\infty$.} 

\vskip 5pt

Recall from Sections~\ref{sec:cohomology} and \ref{sec:GL(N)} that the representation $\Pi$ being cohomological means that there is a highest weight $\mu = (\mu_v)_{v \in \S_\infty}$, with 
$\mu_v = (\mu_{v,1}, \dots, \mu_{v,2n})$ where $\mu_{v,j} \in \Z$ and $\mu_{v,1} \geq \mu_{v,2} \geq \cdots \geq 
\mu_{v,2n},$ etc., and, from (\ref{eqn:j-mu-even}) and (\ref{eqn:pi-infinity-even}) we have
\begin{equation}
\label{eqn:pi-v-explicit}
\Pi_v = J_{\mu_v} = D_{\ell_{v,1}}|\ |^{-{\sf w}/2} \times \dots \times D_{\ell_{v,n}}|\ |^{-{\sf w}/2}.
\end{equation}
Then the central character of $\Pi_v$ is given by
$$
\omega_{\Pi_v} = ({\rm sgn}^{\ell_{v,1} + 1}|\ |^{-{\sf w}})\cdots  ({\rm sgn}^{\ell_{v,n} + 1}|\ |^{-{\sf w}}) = 
{\rm sgn}^{n{\sf w}}|\ |^{-n{\sf w}}.
$$
Now let us invoke the hypothesis that $\Pi$ has an $(\eta, \psi)$-Shalika model. Then $\eta^n = \omega_\Pi$. Hence for all $v \in S_\infty$, there is an $e_v \in \{0,1\}$ such that 
$\eta_v = {\rm sgn}^{\sf w} |\ |^{-{\sf w}} {\rm sgn}^{e_v}$, with $n e_v \equiv 0 \pmod{2}.$ In fact, a stronger statement is true: for all $v \in S_\infty$ one has
\begin{equation}
\label{eqn:eta-v}
\eta_v \ = \ {\rm sgn}^{\sf w} |\ |^{-{\sf w}} .
\end{equation}
To prove (\ref{eqn:eta-v}), note that  by Theorem \ref{thm:shalika}, $\Pi$ is a Langlands functorial transfer of a cuspidal representation of $\GSpin_{2n+1}(\A)$ with central character $\eta$.  Moreover, the lift is strong at the archimedean places, i.e., for each archimedean place, the L-parameter $\phi_v$ of $\Pi_v$ factors through the dual group 
${\rm GSp}_{2n}(\C)$ of $\GSpin_{2n+1}$ with similitude character $\eta_v$. For $v \in S_\infty,$ 
let $(\phi_v, U_v)$ be the $L$-parameter of $\Pi_v$, i.e., $\phi_v$ is a representation of $W_{F_v} = W_\R$ on a 
$2n$-dimensional $\C$-vector space $U_v$. From (\ref{eqn:pi-v-explicit}) we know that 
$$
(\phi_v, U_v) \ = \ \bigoplus_{i=1}^n (\phi_{v,i}, U_{v,i}), 
$$
where each $\phi_{v,i}$ is an irreducible $2$-dimensional representation of $W_{F_v}.$ 
To say that the $L$-parameter $\phi_v$ factors through ${\rm GSp}(2n,\C)$, means that there is a skew-symmetric non-degenerate bilinear form $B_v$ on $U_v$ such that
$$
B_v \in {\rm Hom}_{W_{F_v}}(U_v \otimes U_v, \eta_v). 
$$
Further, from (\ref{eqn:pi-v-explicit}) one knows that 
$$
\phi_{v,j} \ = \ {\rm Ind}_{\C^\times}^{W_\R}\left(re^{i \theta} \mapsto e^{i \ell_{v,j} \theta} \right) \otimes |\ |^{-{\sf w}/2} 
\ =: \ I(\ell_{v,j}) \otimes |\ |^{-{\sf w}/2}. 
$$
In particular, the dual $\phi_{v,j}^{\sf v}$ is $I(\ell_{v,j}) \otimes |\ |^{{\sf w}/2}.$ Hence if $i \neq j$ then $\phi_{v,i}$  is not twist-equivalent to the dual of  $\phi_{v,j}.$  This implies that $B_v = \bigoplus_i B_{v,i}$ and each 
$B_{v,i} := B|_{U_{v,i}}$ is a non-degenerate skew-symmetric bilinear form on $U_{v,i}$: 
$$
B_{v,i} \in {\rm Hom}_{W_{F_v}}(U_{v,i} \otimes U_{v,i}, \eta_v) \ = \ 
{\rm Hom}_{W_{F_v}}(\phi_{v,i} \otimes \phi_{v,i}, \eta_v).
$$
But $B_{v,i}$ is skew-symmetric, hence 
$$
0 \neq B_{v,i} \in {\rm Hom}_{W_{F_v}}( \wedge^2U_{v,i}, \eta_v) = {\rm Hom}_{W_{F_v}}({\rm det}(\phi_{v,i}), \eta_v) = 
{\rm Hom}_{W_{F_v}}({\rm sgn}^{\sf w}|\ |^{-{\sf w}}, \eta_v), 
$$
since ${\rm det}(\phi_{v,i}) = {\rm sgn}^{\ell_{v,i} + 1}|\ |^{-{\sf w}} = {\rm sgn}^{\sf w}|\ |^{-{\sf w}}.$ This proves 
(\ref{eqn:eta-v}). Together with  (\ref{eqn:sigma-omega}), we conclude that
 \[  ({}^\sigma\eta)_v =  \eta_v = \eta_{\sigma^{-1} v} \]
 for all $\sigma \in \autc$.

\vskip 5pt

Next, it follows by Theorem~\ref{thm:shalika} that $L^S(s, \Pi, \wedge^2 \otimes \eta^{-1})$ has a pole at $s=1$, and thus 
$ \Pi^{\sf v} \cong \Pi \otimes \eta^{-1}$. For $\sigma \in {\rm Aut}(\C)$, 
we see, by checking locally almost everywhere, that 
\[  {^{\sigma}}\Pi^{\sf v} \cong {^{\sigma}}\Pi \otimes  {^{\sigma}}\!\eta^{-1}, \]
and thus
\[  L^S(s, {^{\sigma}}\Pi \otimes  {^{\sigma}}\Pi \otimes {^{\sigma}}\!\eta^{-1}) = L^S(s, {^{\sigma}}\Pi, {\rm Sym}^2 \otimes {^{\sigma}}\!\eta^{-1}) \cdot
L^S(s, {^{\sigma}}\Pi, \wedge^2 \otimes {{^{\sigma}}\!\eta^{-1}})  \]
has a pole at $s = 1$. To prove the theorem, we need to show that the ${\rm Sym}^2$ $L$-function does not have a pole at $s = 1$. 

\vskip 5pt

   Let us 
suppose, for the sake of contradiction, that $ L^S(s, {^{\sigma}}\Pi, {\rm Sym}^2 \otimes {^{\sigma}}\!\eta^{-1})$ has a pole at $s  =1$. Then by Theorem \ref{thm:Sym2} , one knows that $ {^{\sigma}}\Pi$ is a Langlands functorial transfer from a cuspidal representation of $\GSpin_{2n}(\A)$ with connected central character ${^{\sigma}}\!\eta$, and this lift is strong at the archimedean places. Fix any $v \in S_\infty$ and put $w = \sigma^{-1} v.$ 
The $L$-parameter of 
${}^\sigma\Pi_v = \Pi_{\sigma^{-1}v} = \Pi_w$ (see (\ref{eqn:n-even-aut-c})) preserves a symmetric non-degenerate bilinear form $C_w$ with similitude character $({}^\sigma\eta)_v = \eta_v = {\rm sgn}^{\sf w}|\ |^{-{\sf w}}$. In this case, we get a non-degenerate symmetric bilinear form $C_{w,i}$ on $U_{w,i}$ and $W_{F_w}$ preserves this form up to similitude character 
${\rm sgn}^{\sf w}|\ |^{-{\sf w}}$, i.e., 
$$
C_{w,i} \in {\rm Hom}_{W_{F_w}}(U_{v,i} \otimes U_{v,i}, \ {\rm sgn}^{\sf w}|\ |^{-{\sf w}}).
$$
But each $\phi_{w,i}$ being irreducible, by Schur's Lemma, we know that the space
$$
{\rm Hom}_{W_{F_w}}(U_{v,i} \otimes U_{v,i}, \ {\rm sgn}^{\sf w}|\ |^{-{\sf w}}) \ = \ 
{\rm Hom}_{W_{F_w}}(\phi_{v,i},  \ \phi_{v,i}^{\sf v} \otimes {\rm sgn}^{\sf w}|\ |^{-{\sf w}}), 
$$
is one-dimensional. Hence, the non-degenerate symmetric form $C_{w,i}$ is a multiple of the non-degenerate 
skew-symmetric form $B_{w,i}$; a contradiction! 
\end{proof}

\begin{rem} {\rm 
The hypothesis that $F$ is totally real  is rather artificial. 
One expects the arithmeticity result to hold even without this hypothesis, however, the above proof would not go through. Suppose, for example, $F$ is an imaginary quadratic extension, then we need to consider cohomological representations of $\GL_{2n}(\C)$. For the infinite place $v$, the parameter of the representation $\Pi_v,$ which is a $2n$-dimensional representation of $W_\C = \C^\times$, is of the form: 
$$
\phi_v = \bigoplus_{j=1}^{2n} (z \mapsto z^{a_j}{\bar{z}}^{b_j}).
$$
(Here the $a_j$ and $b_j$ are half-integers; see Clozel~\cite[p.112]{clozel}.) If $\Pi$ has a Shalika model, then the image of the Langlands parameter $\phi_v$ is inside a split torus in 
${\rm Sp}(2n,\C)$. But this split torus may also be viewed as sitting inside ${\rm SO}(2n,\C)$. Hence, from information of 
$\Pi_\infty = \Pi_v$ it is not possible to deduce that the parameter of ${}^\sigma\Pi_v$ is not of orthogonal type. 
}\end{rem}

\vskip 5pt

\begin{rem} \label{r:pole}
{\rm The proof of Theorem \ref{thm:arithmeticity-shalika} amounts to showing that
\[ \text{$L^S(s,\Pi,\wedge^2\otimes\eta^{-1})$ has a pole at $s=1$} \Longrightarrow
\text{$L^S(s, {}^\sigma\Pi,\wedge^2\otimes {}^\sigma\eta^{-1})$ has a pole at $s=1$} \] 
under the conditions in the theorem. The same proof shows that when $F$ is totally real,   
\[ \text{$L^S(s,\Pi, {\rm Sym}^2\otimes\eta^{-1})$ has a pole at $s=1$} \Longrightarrow
\text{$L^S(s, {}^\sigma\Pi, {\rm Sym}^2\otimes {}^\sigma\eta^{-1})$ has a pole at $s=1$} \] 
when $\Pi$ is cuspidal cohomological.}\end{rem}

\vskip 5pt
\section{Arithmeticity of $\GL(n)/F$-periods for representations of $\GL(n)/E$}
\label{sec:glne-glnf}

The argument of the previous section can be applied to prove the arithmeticity of $\GL_n(F)$-periods for representations of $\GL_n(E)$, where $E$ is a quadratic extension of $F$.

More precisely, let $c$ be the nontrivial element in ${\rm Gal}(E/F)$ and let $\omega_{E/F}$  be the quadratic Hecke character associated to $E/F$ by global class field theory.
 Let $\chi$ be a Hecke character of $\A_E^{\times}$ whose restriction to $\A_F^{\times}$ is equal to $\omega_{E/F}$. Then for $\epsilon = \pm$, we set
\[ \omega_{E/F}^{\epsilon} = \begin{cases} 
1, \text{  if $\epsilon = +$;} \\
\omega_{E/F} \text{  if $\epsilon = -$.} \end{cases}  \qquad    \chi^{\epsilon} = \begin{cases}
1 \text{  if $\epsilon = +$;} \\
\chi \text{  if $\epsilon = -$.}
\end{cases}  \]

For a cuspidal representation $\Pi$ of $\GL_n(\A_E)$ and $\epsilon = \pm$, we shall  consider the period integral
\[  \mathcal{P}^{\epsilon}(\varphi) =  \int_{Z_F(\A_F) \GL_n(F) \backslash \GL_n(\A_F)}
\varphi(h) \cdot \omega_{E/F}^{\epsilon}(\det(h))
\, dh, \]
where $\varphi \in \Pi$.  For the period integral $\mathcal{P}^{\epsilon}$ to have a chance to be nonvanishing, it is necessary that the central character $\omega_{\Pi}$ of $\Pi$ is equal to 
$(\omega_{E/F}^{\epsilon})^n$  when restricted to the center $Z_F(\A_F) = \A_F^{\times}$ of $\GL_n(\A_F)$. 
\vskip 5pt

Associated to $\Pi$ is a pair of partial L-functions $L^S(s, \Pi, {\rm Asai}^{\pm})$, known as the ${\rm Asai}^{\pm}$ (or twisted tensor) L-function (see \cite[Section 7]{ggp}). One has
\[  L^S(s, \Pi, {\rm Asai}^-) =   L^S(s, \Pi \otimes \chi, {\rm Asai}^+) \]
 and
\[  L^S(s, \Pi \times \Pi^c)  = L^S(s, \Pi, {\rm Asai}^+) \cdot L^S(s, \Pi, {\rm Asai}^-) \]
where  $c$ acts on the representations of $\GL_n(\A_E)$ by 
\[  \Pi^c(g)  =  \Pi(c(g)). \]
 
\vskip 5pt

The following theorem is a consequence of the works of many people (Kim-Krishnamurthy \cite{KK1, KK2}, Flicker \cite{F1, F2},  Ginzburg-Rallis-Soudry \cite{grs}).
\vskip 5pt

\begin{thm} \label{thm:asai}
For a cuspidal automorphic representation $\Pi$ of $\GL_n(\A_E)$, the following are equivalent:
\vskip 5pt

\begin{enumerate}
\item[(i)] There is a $\varphi \in\Pi $  such that $\mathcal{P}^{\epsilon}(\varphi) \neq 0$.
 
\item[(ii)] For a sufficiently large finite set $S$ of places of $F$, the partial ${\rm Asai}^{\epsilon}$
L-function $L^S(s, \Pi, {\rm Asai}^{\epsilon})$  has a pole at $s=1$.

\item[(iii)] $\Pi \otimes \chi^{\epsilon \cdot (-1)^{n-1}}$ is the transfer (standard base change) of a globally generic cuspidal automorphic representation $\pi$ of the quasi-split $\U_n(\A)$.
 \end{enumerate}
 Moreover, when these conditions hold, the transfer in (iii) is strong at all archimedean places of $F$, in the sense that it respects L-parameters. 
\end{thm}
\vskip 5pt

One has a local analog of the above global theorem, which is due to the works of many people (A. Kable \cite{kable}, Anandavardhanan-Kable-Tandon \cite{akt},  N. Matringe \cite{m2, m4}):
\vskip 5pt

\begin{thm} \label{thm:asai-local}
Let $v$ be a non-archimedean place of $F$ which is inert in $E$ and let $\Pi_v$ be a generic representation of $\GL_n(E_v)$. Then the following are equivalent:
\begin{enumerate}
\item[(i)] $\Pi_v$ is $(\GL_n(F_v), \omega_{E_v/F_v}^{\epsilon})$-distinguished.

\item[(ii)] The local Asai L-function $L(s, \Pi_v, {\rm Asai}^{\epsilon})$ has an ``exceptional"  pole at $s = 0$.

\item[(iii)] The L-parameter $\phi_v$ of $\Pi_v$ is conjugate-self-dual with sign $\epsilon$ (in the sense of \cite[Section 3]{ggp}).  \end{enumerate}
\end{thm}

\vskip 5pt
 
In analogy with Theorem \ref{thm:arithmeticity-shalika}, one has the following theorem.  \vskip 5pt

\begin{thm} \label{thm:arithmeticity-unitary}
Let $\Pi$ be a cohomological cuspidal automorphic representation of $\GL_n(\A_E)$ which is globally distinguished by $(\GL_n(\A_F), \omega_{E/F}^{\epsilon})$. Assume that one of the following conditions hold:
\vskip 5pt

\begin{enumerate}
\item[(1)] $n$ is odd; or
\item[(2)] $E$ is a totally complex and $F$ is totally real; or
\item[(3)] there is a finite place $v$ of $F$ which is inert in $E$ where $\Pi_v$ is discrete series.
\end{enumerate}
Then, for any $\sigma \in \autc$, ${}^\sigma\Pi$ is a cohomological cuspidal representation 
which is globally distinguished by  $(\GL_n(\A_F), \omega_{E/F}^{\epsilon})$.
\end{thm}

\begin{proof}
This is similar to the proof of Theorem  \ref{thm:arithmeticity-shalika}, exploiting Theorems \ref{thm:asai} and \ref{thm:asai-local}:
\begin{align}
&\text{$\Pi$ is globally distinguished by $(\GL_n(\A_F), \omega_{E/F}^{\epsilon})$} \notag \\
\Longrightarrow &\text{$L^S(s, \Pi, {\rm Asai}^{\epsilon})$ has a pole at $s = 1$} \notag \\
\Longrightarrow &\Pi^c \cong \Pi^{\vee} \notag \\
\Longrightarrow &{}^\sigma\Pi^c \cong {}^\sigma\Pi^{\vee} \notag \\
\Longrightarrow &\text{$L^S_E(s, {}^\sigma\Pi \times {}^\sigma\Pi^c)$ has a pole at $s=1$} \notag \\
\Longrightarrow &\text{$L^S(s, {}^\sigma\Pi, {\rm Asai}^{+})$ or $L^S(s, {}^\sigma\Pi, {\rm Asai}^{-})$ 
has a pole at $s = 1.$} \notag
\end{align}
Suppose that $L^S(s, {}^\sigma\Pi, {\rm Asai}^{-\epsilon})$ has a pole at $s = 1$, rather than $L^S(s,{}^\sigma \Pi, {\rm Asai}^{\epsilon})$, so that ${}^\sigma\Pi$ is distinguished by
 $(\GL_n(\A_F), \omega_{E/F}^{-\epsilon})$.
We shall obtain a contradiction under one of the hypotheses (1), (2) or (3). 
\vskip 5pt

Under hypothesis (1), so that $n$ is odd, we note that the central character of $\Pi$ is equal to $\omega_{E/F}^{\epsilon}$ when restricted to the center of $\GL_n(\A_F)$, whereas that of ${}^\sigma\Pi$ is equal to $\omega_{E/F}^{-\epsilon}$. In particular, one restriction is the trivial character of $\A_F^{\times}$ whereas the other is the quadratic character $\omega_{E/F}$. However, at all finite places, it is clear that the central characters of $\Pi_v$ and ${}^\sigma\Pi_v$ have the same restriction to the center of $\GL_n(F_v)$ since this restriction is at most a quadratic
character. This gives the desired contradiction under hypothesis (1).
\vskip 5pt

Now assume hypothesis (2), so that $E$ is a totally complex extension of the totally real field $F$.
By Theorem \ref{thm:asai}, $\Pi \otimes \chi^{\epsilon \cdot (-1)^n}$ is lifted from $U_n(\A)$ and the lift is strong at archimedean places. Thus for each place $v$ of $E$, the L-parameter $\phi_v$ of $\Pi_v$ is of the form
\[  \phi_v = \bigoplus_i  z^{a_i} \cdot \overline{z}^{-a_i} \]
with $a_1 > a_2 > \cdots > a_n$ half-integers and with each character $z \mapsto z^{a_i} \cdot \overline{z}^{-a_i}$ conjugate-self-dual  with sign $\epsilon$, i.e., 
\[ 2 a_i  = \begin{cases}
\text{even, if $\epsilon = +1$;} \\
\text{odd, if $\epsilon = -1$.} \end{cases} \]
 On the other hand, consider the $L$-parameter $\phi_v'$ of ${}^\sigma \Pi_v$. By Proposition \ref{prop:arch}, ${}^\sigma\Pi_v = \Pi_{\sigma^{-1} \cdot v}$, so that $\phi'_v = \phi_{\sigma^{-1}v}$.
 Thus $\phi'_v$ is the direct sum of characters which are conjugate-self-dual of sign $\epsilon$. But if ${}^\sigma \Pi_v$  is distinguished by
 $(\GL_n(\A_F), \omega_{E/F}^{-\epsilon})$, Theorem \ref{thm:asai} implies that $\phi'_v$ is the direct sum of characters which are conjugate-self-dual with sign $-\epsilon$. This gives the desired contradiction.

\vskip 5pt

Finally, assume hypothesis (3).  For all finite places $v$ of $F$, ${}^\sigma\Pi_v$ is locally distinguished by $(\GL_n(F_v), \omega_{E_v/F_v}^{-\epsilon})$ and so its L-parameter $\phi'_v$
is conjugate-self-dual with sign $-\epsilon$. On the other hand, the L-parameter
$\phi_v$ of $\Pi_v$ is conjugate self-dual of sign $\epsilon$. When $\Pi_v$ is discrete series, $\phi'_v = {}^\sigma \phi_v$ up to the quadratic character $x \mapsto \sigma( |x|_{E_v}^{1/2})/ |x|_{E_v}^{1/2}$ of $W_{E_v}^{ab} \cong E_v^{\times}$. Observe that this character is trivial on $F_v^{\times}$, so it is conjugate orthogonal in the sense of \cite[Section 3]{ggp}. In particular, $\phi_v$ and $\phi'_v$ are conjugate-self-dual of the same sign; this gives  the desired contradiction when $\Pi_v$ is discrete series.
 \end{proof}

\vskip 5pt

\section{Arithmeticity of $\GL_{n-1}$ periods on ${\rm GL}_n \times \GL_{n-1}$}
\label{sec:gln-gl(n-1)}

In this section, we 
consider the $\GL_{n-1}$-period for cuspidal representations of $\GL_n \times \GL_{n-1}$ over $\Q$. 
This context is a very nice generalization of the example in subsection~\ref{sec:gl2-gl1} where we studied 
$(\GL_1, \chi)$-periods for representations of $\GL_2$. 
The nonvanishing of periods is equivalent to a certain central $L$-value being nonzero. If we further impose the condition that the central value is a critical value then an appropriate algebraicity theorem for this critical value gives arithmeticity. The situation is analogous to Gross's conjecture concerning order of vanishing of critical motivic $L$-values as discussed in the introduction. 

\begin{thm} 
\label{thm:ggp-gln}
Let $\Pi$ be a cohomological cuspidal automorphic representation of 
$\GL_{n}(\A),$ say $\Pi \in {\rm Coh}(\GL_n, \mu)$. Here $\A$ is the adele ring of $\Q$.
Similarly, let $\Sigma \in  {\rm Coh}(\GL_{n-1}, \lambda).$ Suppose that $\Pi \otimes \Sigma$, as a representation of 
$(\GL_{n} \times \GL_{n-1})(\A),$ has a non-vanishing period with respect to the diagonally embedded subgroup 
$\GL_{n-1}(\A)$. Suppose further that the coefficient systems $E_{\mu}$ and $E_{\lambda}$ satisfy: 
$$
{\rm Hom}_{\GL_{n-1}}(E_{\mu} \otimes E_{\lambda}, \ 1\!\!1) \neq 0.
$$
Then for  any $\sigma \in {\rm Aut}(\C)$, the representation 
${}^{\sigma}\Pi \times {}^{\sigma}\Sigma$ also has a non-vanishing period with respect to 
$\GL_{n-1}(\A)$ under the assumption that \cite[Hypothesis 3.10]{raghuram-imrn} holds. 
\end{thm}

\begin{proof} 
Every step of the proof is a suitable generalization of the proof of arithmeticity of $(\GL_1, \chi)$ for representations of 
$\GL_2$ as in subsection~\ref{sec:gl2-gl1}.

To begin, the generalization of Proposition~\ref{prop:gl2-characterize} goes like this: 
$\Pi \otimes \Sigma$ as a representation of 
$(\GL_{n} \times \GL_{n-1})(\A)$ has a non-vanishing $\GL_{n-1}(\A)$ period if and only if 
$L(\tfrac12, \Pi \times \Sigma) \neq 0.$  This follows from using the integrals studied by Jacquet, Piatetskii-Shapiro and Shalika \cite{jpss} as follows. For cusp forms $\phi \in V_\Pi$ and $\phi' \in V_\Sigma$, define 
$$
\ell(s, \phi, \phi') = \int_{\GL_{n-1}(\Q)Z_{n-1}(\A)\backslash\GL_{n-1}(\A)} 
\phi\left(\begin{array}{ll} g & 0 \\ 0 & 1\end{array}\right) \phi'(g) |{\rm det}(g)|^{s - \tfrac12}\, dg, 
$$ 
where $Z_{n-1}$ is the center of $\GL_{n-1}.$  
Our assumption on $\Pi \otimes \Sigma$ is that $\ell(\tfrac12, \phi, \phi') \neq 0$ for some $\phi$ and $\phi'.$
Let $W_\phi$ and $W_{\phi'}$ be the corresponding Whittaker vectors; we may and will take $\phi$ and $\phi'$ so that $W_\phi$ and $W_{\phi'}$ are pure-tensors: $W_\phi = \otimes W_p$ and 
$W_{\phi'} = \otimes W'_p.$ The unfolding argument gives $\ell(s, \phi, \phi') = Z(s, W_\phi, W_{\phi'})$ 
for $\Re(s) \gg 0$, where 
$$
Z(s, W_\phi, W_{\phi'}) = \int_{N_{n-1}(\A)\backslash\GL_{n-1}(\A)} 
W_\phi\left(\begin{array}{ll} g & 0 \\ 0 & 1\end{array}\right) W_{\phi'}(g) |{\rm det}(g)|^{s - \tfrac12}\, dg; 
$$
here $N_{n-1}$ is the subgroup of all upper triangular unipotent elements in $\GL_{n-1}$.  
The analogue of (\ref{eqn:gl2-zeta}) takes the form: 
\begin{equation}
\label{eqn:gln-zeta}
\ell(s,\phi,\phi') = \left(\prod_{p \in S} \frac{Z(s,W_p, W'_p)}{L_p(s,\Pi_p \otimes \Sigma_p)} \right) \cdot 
L(s, \Pi \otimes \Sigma). 
\end{equation}
Using \cite[Theorem 2.7]{jpss} we get that both sides and especially both the factors on the right hand side are entire functions. Evaluating at $s=1/2$ gives
\begin{equation}
\label{eqn:gln-characterize}
\Pi \otimes \Sigma \ \mbox{has a non-vanishing $\GL_{n-1}(\A)$ period} \  
\ \Longleftrightarrow \ L(\tfrac12, \Pi \times \Sigma) \neq 0.
\end{equation}

Next, the hypothesis that $\Pi \in {\rm Coh}(\GL_n, \mu)$ and $\Sigma \in  {\rm Coh}(\GL_{n-1}, \lambda)$ 
puts us in an arithmetic context, however, this doesn't guarantee that $s = \tfrac12$ is critical. With the foresight of wanting to appeal to algebraicity results, we impose the condition: 
$$
{\rm Hom}_{\GL_{n-1}}(E_{\mu} \otimes E_{\lambda}, \ 1\!\!1) \neq 0.
$$
It was observed by Kasten and Schmidt \cite[Theorem 2.3]{kasten-schmidt} that this condition implies 
$s = 1/2$ is critical for the Rankin-Selberg $L$-function $L(s, \Pi \times \Sigma).$ The same condition is also needed 
for an algebraicity result for critical values due the second author; see \cite{raghuram-imrn}. We get the analogue of 
(\ref{eqn:gl2-global-finite}) which looks like
\begin{equation}
\label{eqn:gln-global-finite}
L(\tfrac12, \Pi \otimes \Sigma) \neq 0 \ \Longleftrightarrow \ L_f(\tfrac12, \Pi \otimes \Sigma) \neq 0.
\end{equation}

 Under an additional nonvanishing hypothesis involving only representations at infinity as in 
\cite[Hypothesis 3.10]{raghuram-imrn}, the main result of that paper, \cite[Theorem 1.1]{raghuram-imrn}, says that 
$$
\sigma\left(\frac{L_f( \tfrac12,\Pi \times \Sigma)}
{p^{\epsilon}\, (\Pi)p^{\eta}\, (\Sigma)\mathcal{G}(\omega_{\Sigma_f})\, p_{\infty}(\mu,\lambda)}\right)
\  = \  
\frac{L_f(\tfrac12, \Pi^{\sigma} \times \Sigma^{\sigma})}
{p^{\epsilon}(\Pi^{\sigma})\, p^{\eta}(\Sigma^{\sigma})\, \mathcal{G}(\omega_{\Sigma_f^{\sigma}})\,
p_{\infty}(\mu,\lambda)},
$$
where $p^{\epsilon}\, (\Pi)$ and $p^{\eta}\, (\Sigma)$ are nonzero complex numbers, $\mathcal{G}(\omega_{\Sigma_f})$ is the Gauss sum of the central character of $\Sigma$, and $p_{\infty}(\mu,\lambda)$ is a nonzero complex number determined by $\mu$ and $\lambda$. The analogue of (\ref{eqn:gl2-gross})  follows easily:  
\begin{equation}
\label{eqn:gln-gross}
L_f(\tfrac12, \Pi \otimes \Sigma) \neq 0 \  \Longleftrightarrow \ 
L_f(\tfrac12, {}^\sigma\Pi \otimes {}^\sigma\Sigma) \neq 0.
\end{equation}

Arithmeticity follows from first applying (\ref{eqn:gln-characterize}), (\ref{eqn:gln-global-finite}) and (\ref{eqn:gln-gross}) 
to $\Pi \otimes \Sigma$ and then applying (\ref{eqn:gln-global-finite}) and (\ref{eqn:gln-characterize}) to 
 ${}^\sigma\Pi \otimes {}^\sigma\Sigma$
\end{proof}

\begin{rem}{\rm The hypothesis on the coefficient systems as in Theorem~\ref{thm:ggp-gln}, which is itself a nonvanishing period like condition, is crucial for the methods of \cite{raghuram-imrn} to apply. Let us note that 
it is possible to have a pair of cohomological representations $\Pi$ and $\Sigma$ for which $s = 1/2$ is critical but for which that condition on the coefficients is not satisfied. For example, take $n=3$, 
$\mu = (0,0,0)$ and $\lambda = (1,-1)$; then $E_{\mu}$ is the trivial representation of $\GL_3$. Take 
$\Pi \in {\rm Coh}(\GL_3, \mu)$ and $\Sigma \in  {\rm Coh}(\GL_2, \lambda)$. Then, we leave it to the reader to check that $s = 1/2$ is critical for $L(s, \Pi \times \Sigma)$, but 
${\rm Hom}_{\GL_2}(E_{\mu} \otimes E_{\lambda}, \ 1\!\!1) = 0.$ Now in such a situation, suppose 
the representation $\Pi \times \Sigma$ of $\GL_3 \times \GL_2$ has a nonvanishing $\GL_2$ period, then the above proof is not applicable; however, we still believe that one should have arithmeticity. 
}\end{rem}

\begin{rem}{\rm 
In a certain work in progress \cite{raghuram-preprint}, the second author is studying algebraicity theorems for critical values of $L$-functions for $\GL_n \times \GL_{n-1}$ over any number field. This would then generalize 
Theorem~\ref{thm:ggp-gln} from $\Q$ to any number field. 
}\end{rem}

\begin{rem}{\rm
The assumption  \cite[Hypothesis 3.10]{raghuram-imrn} is a certain limitation of the technique used in that paper. We note that this hypothesis is of a purely local nature and depends only the representations $\Pi_\infty$ and $\Sigma_\infty$ at infinity. For $n=2$ the validity of this hypothesis follows from an explicit calculculation; see \cite{raghuram-tanabe}; it is this calculation that gives the term $(2 \bfgreek{pi} i)^{d_\infty}$ in Proposition~\ref{prop:manin-shimura}. 
For $n=3$ the validity of the hypothesis has been proved by Kasten and Schmidt~\cite{kasten-schmidt}. 
}\end{rem}

\section{Arithmeticity of $\GL_{n} \times \GL_n$ periods for cusp forms on ${\rm GL}_{2n}$}
\label{sec:gl2n}

In this section, we discuss yet another generalization of the example in subsection~\ref{sec:gl2-gl1} where we studied 
$(\GL_1, \chi)$-periods for representations $\pi$ of $\GL_2$. Indeed, in that example, we could have carried through the entire discussion by replacing $\pi$ by $\pi \otimes \chi$ and taking the trivial character of $H = \GL_1 \times \GL_1$ sitting as the diagonal torus in $\GL_2.$ (This imposes the condition that the central character of $\pi\otimes \chi$ is trivial.) 

Now we take $G = \GL_{2n}$ over a totally real number field $F$. Take $H = \GL_n \times \GL_n$ sitting as block diagonal matrices in $G$. Let $\Pi$ be a cuspidal automorphic representation of $G(\A_F)$ which admits a Shalika model (the analogue of the triviality of the central character mentioned above). 
We would like to analyze arithmeticity for the periods: 
$$
\ell(\phi) \ := \ \int_{[H]} \phi\left(\begin{array}{ll} g_1 & 0 \\ 0 & g_2 \end{array}\right)\, dg_1dg_2, \ \ \ \phi \in V_\Pi, 
$$
where $[H] = Z_G(\A_F)H(F)\backslash H(\A_F).$ Arithmeticity in this context follows from certain zeta integrals 
studied by Jacquet-Shalika \cite{jacquet-shalika} and Friedberg-Jacquet \cite{friedberg-jacquet}, and an algebraicity result due to Grobner and the second author \cite{grobner-raghuram-2}. 

\begin{thm} 
\label{thm:gl2n}
Let $\Pi$ be a cohomological cuspidal automorphic representation of 
$\GL_{2n}(\A_F)$ where $F$ is a totally real number field. Suppose that  
\begin{enumerate}
\item $\Pi$ has a nonvanishing $H$-period;
\item the point $s = \tfrac12$ is critical for $L(s, \Pi)$. 
\end{enumerate}
Then for  any $\sigma \in {\rm Aut}(\C)$, the representation 
${}^{\sigma}\Pi$ also has a non-vanishing $H$-period.
\end{thm}

\begin{proof} 
Again, we follow the proof of arithmeticity of $(\GL_1, \chi)$ for representations of 
$\GL_2$ as in subsection~\ref{sec:gl2-gl1}. 
\vskip 5pt

To begin, by a theorem of Friedberg-Jacquet \cite{friedberg-jacquet},  $\pi$ has nonzero $H$-period if and only if $\pi$ admits a Shalika model and $L(1/2, \pi) \ne 0$. 
For a cusp form $\phi \in V_\Pi$ define 
$$
\ell(s, \phi) = \int_{H(F)Z_{2n}(\A_F)\backslash H(\A_F)} 
\phi\left(\begin{array}{ll} g_1 & 0 \\ 0 & g_2\end{array}\right)  
\left| \frac{{\rm det}(g_1)}{{\rm det}(g_2)} \right|^{s - \tfrac12}\, dg_1dg_2, 
$$ 
where $Z_{2n}$ is the center of $\GL_{2n}.$  
Our assumption on $\Pi$ is that $\ell(\tfrac12, \phi) \neq 0$ for some $\phi.$  
Let $\S_\phi$  be the corresponding vector in the Shalika model of $\Pi$; as before, we may and will take $\phi$ so that 
$\S_\phi$ is a  pure-tensor: $\S_\phi = \otimes \S_p.$ (For details concerning Shalika models and related matters 
we refer the reader to \cite{grobner-raghuram-2}, and recommend that any serious reader of this section should have that paper by one's side.)

An unfolding argument (\cite[Proposition 3.1.5]{grobner-raghuram-2})  
gives $\ell(s, \phi) = Z(s, \S_\phi)$ where 
$$
Z(s, S_\phi) = \int_{\GL_n(F) \backslash\GL_{n}(\A_F)} 
\S_\phi\left(\begin{array}{ll} g & 0 \\ 0 & 1\end{array}\right) |{\rm det}(g)|^{s - \tfrac12}\, dg. 
$$
The analogue of (\ref{eqn:gln-zeta}) takes the form: 
\begin{equation}
\label{eqn:gl2n-zeta}
\ell(s,\phi) = \left(\prod_{p \in S} \frac{Z(s, \S_p)}{L_p(s, \pi_p)} \right) \cdot 
L(s, \pi). 
\end{equation}
Using \cite[Proposition 3.3.1]{grobner-raghuram-2} we get that the left hand side and both the factors on the right hand side are entire functions. Evaluating at $s=1/2$ gives
\begin{equation}
\label{eqn:gl2n-characterize}
\Pi \ \mbox{has a non-vanishing $\GL_{n}(\A_F)$ period} \  
\ \Longleftrightarrow \ L(\tfrac12, \Pi) \neq 0.
\end{equation}

Next, the hypothesis that $\Pi \in {\rm Coh}(\GL_n, \mu)$  
puts us in an arithmetic context, however, as before, this doesn't guarantee that $s = \tfrac12$ is critical. 
So, we now need the assumption that $s =\tfrac12$ is critical for $L(s, \Pi)$. (In the $\GL_n \times \GL_{n-1}$ case, we needed a stronger condition on the coefficient system, but in the current context  
\cite[Proposition 6.3.1]{grobner-raghuram-2} guarantees that.)  The analogue of 
(\ref{eqn:gln-global-finite}) is:
\begin{equation}
\label{eqn:gl2n-global-finite}
L(\tfrac12, \Pi) \neq 0 \ \Longleftrightarrow \ L_f(\tfrac12, \Pi) \neq 0.
\end{equation}

Under the assumption that $\Pi \in {\rm Coh}(\GL_{2n}, \mu)$ has a Shalika model, the algebraicity result in  
\cite[Theorem 7.1.2]{grobner-raghuram-2} says 
$$
\sigma\left(\frac{L_f( \tfrac12,\Pi \otimes \chi)}
{\omega^{\epsilon_\chi}\, (\Pi) \mathcal{G}(\chi)\, \omega_{\infty}(\mu)}\right)
\  = \  
\frac{L_f(\tfrac12, {}^{\sigma} \Pi \otimes {}^\sigma\chi)}
{\omega^{\epsilon_{{}^{\sigma}\chi}}({}^{\sigma}\Pi)\, \mathcal{G}({}^{\sigma}\chi)\, \omega_{\infty}(\mu)},
$$
where $\chi$ is an algebraic Hecke character, $\epsilon_\chi$ its parity, $\mathcal{G}(\chi)$ its Gau\ss~sum, 
$\omega^{\epsilon_\chi}\, (\Pi)$ is a nonzero complex number, and $\omega_{\infty}(\mu)$ is a nonzero complex number determined by $\mu.$  
The analogue of (\ref{eqn:gln-gross})  follows easily:  
\begin{equation}
\label{eqn:gl2n-gross}
L_f(\tfrac12, \Pi) \neq 0 \  \Longleftrightarrow \ 
L_f(\tfrac12, {}^\sigma\Pi) \neq 0.
\end{equation}

Arithmeticity follows from first applying (\ref{eqn:gl2n-characterize}), (\ref{eqn:gl2n-global-finite}) and (\ref{eqn:gl2n-gross}) to $\Pi$ and then applying (\ref{eqn:gl2n-global-finite}) and (\ref{eqn:gl2n-characterize}) to ${}^\sigma\Pi.$
\end{proof}

\section{Arithmeticity for Classical Groups} \label{S:classical}
 
In this section, we consider the possibility of extending Theorem \ref{thm:clozel}  for $\GL(N)$ to the case of classical groups. By the recent work \cite{arthur-book} of Arthur and others,
one now has a classification of square-integrable automorphic representations for quasi-split classical groups, in terms of automorphic representations of $\GL(N)$.
 In view of this, it is natural to ask if arithmeticity results for $\GL(N)$ can be transferred to these classical groups.

\vskip 5pt

More precisely, let $G$ be a quasi-split symplectic, special orthogonal or unitary group over the number field $F$ and let $\pi$ be a cuspidal automorphic representation of $G(\A_F)$. By Arthur \cite{arthur-book}, one can attach a discrete  $A$-parameter to $\pi$ and this is a multiplicity-free formal sum 
\[ \Psi=  \Pi_1 \boxtimes S_{r_1} \boxplus \cdots \boxplus \Pi_k \boxtimes S_{r_k}, \]
where $\Pi_i$ is a cuspidal automorphic representation of $\GL(n_i)$ (over $F$ or a quadratic extension $E$) satisfying some symmetry conditions and $S_{r_i}$ is the $r_i$-dimensional irreducible representation of $\SL_2(\C)$. 
Moreover, the set of all $\pi$'s with a given discrete $A$-parameter $\Psi$ is a full near equivalence class in the automorphic discrete spectrum of $G$.  If $r_i = 1$ for all $i$,  
$\Psi$ is called a tempered $A$-parameter. 
\vskip 5pt

Now suppose further that $\pi$ is cohomological.
Recall from Proposition \ref{prop:generalG} that for any $\sigma \in \autc$, there is a square-integrable automorphic representation $\tau_{\sigma}$ of $G(\A_F)$ such that
\[  \tau_{\sigma, f} \cong {}^\sigma\pi_f. \]
Note that, for fixed $\sigma \in \autc$, $\tau_{\sigma}$ may not be uniquely determined, but any two such candidates are nearly equivalent to each other and  thus have the same $A$-parameter. 
 It is thus natural to ask:
 \vskip 5pt
 
 \noindent{\bf Question:} How is the $A$-parameter of $\tau_{\sigma}$  related to that of $\pi$?
\vskip 5pt

In the remainder of this section, we shall consider this question for the {\em quasi-split} classical groups $G$ when the $A$-parameter $\Psi$ of $\pi$ is {\em tempered}.  In this case, the $A$-parameter $\Psi$ of $\pi$  is equal to the $L$-parameter of $\pi$. We may also regard $\Psi$   as the  representation 
\[  \Pi = \boxplus_i \Pi_i := {\rm Ind}_P^{\GL(N)} (\otimes_i \Pi_i)  \quad \text{of $\GL(N)$.} \]
One knows moreover that this induced representation is irreducible. 
 In the following, we shall use $\Psi$ and $\Pi$ interchangeably for the $A$-parameter of $\pi$. 
 In particular, for each $v$, the $L$-parameter of $\pi_v$ is precisely the L-parameter of  the generic representation  $\Psi_v = \Pi_v =  \Pi_{1,v} \boxplus \cdots \boxplus \Pi_{k,v}$.  
 
\vskip 5pt

Let us also explicate the symmetry condition satisfied by the summands $\Pi_i$ in $\Psi$  in the various cases:
\vskip 5pt

\begin{itemize}
\item if $G = \Sp(2n)$ or $\SO(2n)$, then $L^S(s, \Pi_i, {\rm Sym^2})$ has a pole at $s=1$ for each $i$;
\item if $G = \SO(2n+1)$, then $L^S(s, \Pi_i, \wedge^2)$ has a pole at $s = 1$;
\item if $G = \U(n)$ , then $L^S(s, \Pi_i , {\rm Asai}^{(-1)^{n-1}})$ has a pole at $s =1$.
\end{itemize}
\vskip 5pt

\begin{rem} \label{rem:u(n)}
{\rm Henceforth, when $G  = \U(n)$, we shall assume that the underlying Hermitian space is defined with respect to a totally complex quadratic extension $E$ of the totally real base field $F$. Moreover, the target group of the functorial lifting  is $\GL(N)$ over the CM field $E$. 
In the rest of this section, in order to simplify the exposition, we shall only give proofs for the symplectic and orthogonal groups, even though the results are stated for $\U(n)$ as well.}
\end{rem}

\vskip 5pt
It is natural  to first investigate if the functorial transfer of unramified representations from classical groups to $\GL(N)$ is $\autc$-equivariant. 
\vskip 5pt

\begin{lemma}  \label{lemma:unram}
Let $k$ be a p-adic field. Let $G$ be an unramified classical group over $k$ and for an unramified representation $\pi$ of $G(k)$, let $\Sigma(\pi)$ be its functorial transfer to the appropriate $\GL(N)$. 
\vskip 5pt
\noindent (i) If $G = \SO(2n+1)$, $\Sp(2n)$ or $\U(n)$, then 
\[  {}^\sigma \Sigma(\pi)  = \Sigma({}^\sigma \pi). \]
\vskip 5pt

\noindent (ii) If $G = \SO(2n)$, then
\[  {}^\sigma (\Sigma(\pi) \otimes  |\ |^{-1/2} ) \otimes  |\ |^{1/2}  = \Sigma({}^\sigma\pi ).   \]
 \end{lemma}

\begin{proof}
Let $T \subset B \subset G$ be a maximal torus contained in a Borel subgroup of $G$, both defined over $k$.   
Similarly, let $T^* \subset B^* \subset \GL(N)$, where $\GL(N)$ is the target of the functorial transfer from $G$. Now Langlands functoriality gives rise to a transfer $\chi \mapsto \chi^*$ from the set of unramified characters of $T$ to that of $T^*$. Moreover, this transfer from $T$ to $T^*$ is $\autc$-equivariant.
\vskip 5pt

Now  $\pi \subset I(\chi) := {\rm Ind}_B^G \chi$ for some unramified character $\chi$ of $T$ (where the induction is normalized here).  The functorial transfer of $I(\chi)$ is then equal to  $I(\chi^*)$.
 For $\sigma \in \autc$, ${}^\sigma \pi$ is still an unramified representation and 
\[  {}^\sigma \pi \subset {}^\sigma I(\chi) = I({}^\sigma \chi \cdot  \delta_{G, \sigma}) \]
where $\delta_{G,\sigma} = {}^\sigma \delta_B^{1/2} / \delta_B^{1/2}$ and $\delta_B^{1/2}$ is the modulus character. 
\vskip 5pt

Thus, we are interested in whether 
\[   I({}^\sigma \chi \cdot  \delta_{G, \sigma}) \mapsto I({}^\sigma \chi^* \cdot \delta_{\GL(N),\sigma}) \]
under functorial transfer. Since ${}^\sigma \chi \mapsto {}^\sigma \chi^*$, it suffices to verify whether  $\delta_{G,\sigma} \mapsto \delta_{\GL(N),\sigma}$.
\vskip 5pt
   Now we note:
\begin{itemize}
\item when $G= \Sp(2n)$, $\SO(2n)$, $\U(2n+1)$ or $\GL(2n+1)$, $\delta_B^{1/2}$ takes value in $\Q^{\times}$ and so $\delta_{G,\sigma}=1$.
\item when $G = \SO(2n+1)$,  $\U(2n)$ or $\GL(2n)$,  $\delta_B^{1/2}$ takes value in $q^{\frac{1}{2} \Z}$ where $q$ is the size of the residue field of $k$, so that $\delta_{G,\sigma}$ is a quadratic character (possibly trivial).
\end{itemize}
In the context of (i),  one sees that  the transfer of $\delta_{G,\sigma}$ to $T^*$ is $\delta_{\GL(N),\sigma}$. On the other hand, when $G = \SO(2n)$, $\delta_{G,\sigma}$ is trivial whereas $\delta_{\GL(2n),\sigma}$ is not necessarily trivial.  A short computation then gives the result in (ii).
\end{proof}

\vskip 10pt
Lemma \ref{lemma:unram} already implies that when $G \ne \SO(2n)$, the transfer (in the sense of Arthur)  of $\tau_{\sigma,f}$ to $\GL(N)$ is nearly equivalent to the abstract representation ${}^\sigma \Pi_f$. 
Next, we note the following crucial proposition:
\vskip 5pt

\begin{prop}  \label{prop:cohom}
Let $G$ be a classical group over a totally real $F$.
Assume that $\pi$ is a cohomological cuspidal representation of $G(\A)$ with tempered $A$-parameter. Then for all infinite places $v$, $\pi_v$ is ``as close to being a discrete series representation as possible". More precisely,
\begin{itemize}
\item[(i)]  if $G(F_v)$ has discrete series representations, then $\pi_v$ is discrete series. This is the case precisely when $G(F_v)$ is not of the form $\SO(2a+1, 2b+1)$ with $a,b \in \Z$. 
\vskip 5pt

\item[(ii)]  If  $G(F_v) = \SO(2a+1,2b+1)$, then $\pi_v = {\rm Ind}_P^G  \chi \boxtimes \pi_{0,v}$ where $P$ is a maximal parabolic subgroup with Levi factor $\GL_1(F_v) \times \SO(2a, 2b)$, $\chi = 1$ or $sgn$, and $\pi_{0,v}$ is a discrete series representation of $\SO(2a,2b)$. 
\end{itemize}
In particular, $\pi_v$ is tempered in all cases. When $G \ne \SO(2n)$, the functorial transfer of $\pi_v$ to $\GL(N)$ is a cohomological representation.
 
\end{prop}
\vskip 5pt
 
\begin{proof}
We shall treat only symplectic and orthogonal groups; the case of unitary groups is similar. 
We make the following observations:
\vskip 5pt

\begin{itemize}
\item[(a)]  Since $\pi$ is cohomological, the infinitesimal character of $\pi_v$ (for  any infinite place $v$) is (strongly) regular and integral. 
\vskip 5pt

\item[(b)]  Since $\pi$ has tempered $A$-parameter, the results of Luo-Rudnick-Sarnak \cite{LRS} imply that $\pi_v$ is ``close to being tempered". 
\end{itemize}
Now we can express $\pi_v$ as a quotient of an induced representation  $I =  {\rm Ind}_{P_v}^{G_v}  \Sigma$, where
\begin{itemize}
\item $P_v$ is a parabolic subgroup with Levi factor 
\[ \GL_1(F_v)^a \times \GL_2(F_v)^b \times G_{0,v} \]
where $G_{0,v}$ is a classical group of the same type as $G_v$;

\vskip 5pt

\item The representation $\Sigma$ has the form
\[  \Sigma = \left( \bigotimes_{i=1}^a \chi_i |-|^{s_i}  \right)  \bigotimes
\left( \bigotimes_{j =1}^b \tau_j |-|^{t_i} \right) \bigotimes  \pi_{0,v} \]
with $\chi_i$ either the trivial or the sign character, $\tau_j$ are discrete series representations of $\GL_2(\R)$ whose central character are trivial or sign, $\pi_{0,v}$ is a discrete series representation of $G_{0,v}$ and $s_i, t_j \in \C$.
\end{itemize}

The integrality of the infinitesimal character of $\pi_v$ implies that the numbers $s_i$ and $t_j$ are half-integers. The ``close to temperedness" of $\pi_v$ in (b) above implies that the numbers $s_i$ and $t_j$ are close to the imaginary axis. Taken together, they imply that $s_i = t_j = 0$. 
In particular, we deduce that $\pi_v$ is tempered. 
\vskip 5pt

 We have yet to make use of the regularity condition in (a). Let us fix a maximal torus $T_v$ of $G_v$ with associated dual complexified Lie algebra $\mathfrak{t}^*$. The character group $X(T_v)$ endows $\mathfrak{t}^*$ with  an integral structure. The infinitesimal character of $\pi_v$ can be regarded as an element in $\mathfrak{t}^*$ up to conjugacy by the absolute Weyl group $W(G_v,T_v)$. Now we observe:
 \begin{itemize}
 \item there are no $\GL_2$-factors in $P_v$.
 \vskip 5pt
 
 To see this, note that if  there were a $\GL_2$-factor in $P_v$, then the infinitesimal character of $\pi_v$ will have the form $( \cdots,  k, -k, \cdots ),$ with $k$ a half integer. But for the classical groups, there is a nontrivial element of the absolute Weyl group which fixes such an element, namely ``exchanging the coordinates $k$ and $-k$, followed by changing the signs of  both coordinates".  This contradicts regularity. 
\vskip 5pt

\item there is at most one $\GL_1$-factor in $P_v$. Indeed, there can be a $\GL_1$-factor if and only if $G_v = \SO(2a+1, 2b+1)$. 
\vskip 5pt

 To see this, note that if there were two $\GL_1$-factors in $P_v$, then the infinitesimal character of $\pi_v$ will have the form $( \cdots, 0, 0 , \cdots )$ which is fixed by a non-trivial element of the absolute Weyl group.  Further, if there is a $\GL_1$-factor, then the infinitesimal character has the form $( \cdots, 0, \cdots )$ and this is fixed by a nontrivial element of  the absolute Weyl group of type $B$ and $C$, namely ``changing the sign  a coordinate". Thus, a $\GL_1$-factor can only occur in type $D$, so that $G = \SO(a,b)$ with $a+b$ even. If a $\GL_1$-factor does occur, then $\SO(a-1, b-1)$ must have a discrete series representation, so that $a$ and $b$ must both be odd.
Conversely, if $a$ and $b$ are both odd, there must be a $\GL_1$-factor, since $\SO(a,b)$ does not have discrete series representations.
\end{itemize}
\vskip 5pt

Summarizing the above observations, we see that unless $G_v = \SO(2a+1, 2b+1)$, there are no $\GL_1$ or $\GL_2$ factors in $P_v$, so that $G_{0,v} = G_v$ and $\pi_v$ is discrete series. When $G_v = \SO(2a+1, 2b+1)$, $G_v$ does not have discrete series and $\pi_v$ is of the form given in (ii).  

\vskip 5pt

 To prove the last assertion of the proposition, let us explicate the discrete series $L$-parameter $\Psi_v$ of $\pi_v$ when $G \ne \SO(2n)$:
\vskip 5pt

\begin{itemize}
\item if $G = \Sp(2n)$, then  $\Psi_v = \oplus_{j=1}^n \phi_{j} \oplus \chi$, where $\chi =1$ or ${\rm sign}$, and the
$\phi_j$'s are  pairwise distinct orthogonal representations of the Weil group $W_{\R}$ of $\R$, which correspond to  discrete series representations of $\GL_2(\R)$ with central character ${\rm sign}$. Moreover,  $(-1)^n \cdot  \chi(-1) = 1$. 
 \vskip 5pt
 
 \item if $G = \SO(2n+1)$, then $\Psi_v = \oplus_{j=1}^n \phi_{j} $ where the $\phi_{j}$'s are pairwise distinct symplectic representations of $W_{\R}$, which correspond to discrete series representations of $\GL_2(\R)$ with trivial central character.
 \vskip 5pt
 
 \item if $G = \U(n)$, then $\Psi_v = \oplus_{j=1}^n \chi_j$, where the $\chi_j$'s are  conjugate dual character of $W_{\C} = \C^{\times}$ of the form $\chi_j(z) = (z/\overline{z})^{\frac{a_j}{2}}$ with 
 $a_j \equiv n+1 \mod 2$. 
 \end{itemize}
 \vskip 5pt
 
 \noindent From this and the description of cohomological representations of $\GL(N)$ given in Section \ref{sec:GL(N)}, we observe that the representation $\Pi_v$ with $L$-parameter $\Psi_v$ is 
 cohomological. The proposition is proved.
 \end{proof}
 \vskip 5pt
 
Recall that  the A-parameter of $\pi$ is $\Pi: = \Pi_1 \boxplus 
\cdots \boxplus \Pi_k$.  As noted by Clozel,  however, it is better to work with a Tate-twisted isobaric sum 
$\stackrel{T}{\boxplus}$: 
\[ \Pi = \boxplus_i \Pi_i =  \ \stackrel{T}{\boxplus}_i  \Sigma_i \]
where
\[   \Sigma_i =  \Pi_i \otimes |\ |^{\frac{n_i-N}{2}}. \]
Then the following corollary  follows from Proposition \ref{prop:cohom}:
\vskip 5pt

\begin{cor}  \label{cor:coho-classical}
(i) Suppose that  $G = \Sp(2n)$, $\SO(2n+1)$ or $\U(n)$ 
 then $\Pi = \ \stackrel{T}{\boxplus}_i \Sigma_i$ is a cohomological representation of $\GL(N)$ and for each $i$,  $\Sigma_i$ is a cohomological cuspidal representation of $\GL(n_i)$. 
\vskip 5pt

(ii) When $G = \SO(2n)$, then $\Pi \otimes |\ |^{-1/2}$ is an algebraic representation of $\GL(N)$ in the sense of \cite{clozel}, and for each $i$,  $\Sigma_i \otimes |\ |^{-1/2}$ is an algebraic representation of $\GL(n_i)$. 
\end{cor}

\vskip 5pt
Now we come to the main result of this section.
\vskip 5pt

\begin{thm} \label{thm:classical}
Let $F$ be totally real and let $G$ be a quasi-split classical group over $F$ with $G = \Sp(2n)$, $\SO(2n+1)$ or $\U(n)$ (c.f. Remark \ref{rem:u(n)}).
Let $\pi$ be a cohomological cuspidal representation of   $G$ with a tempered $A$-parameter 
$\Psi = \boxplus_i \Pi_i = \ \stackrel{T}{\boxplus}_i \Sigma_i$. Then we have:
\vskip 5pt

\noindent (i) For $\sigma \in \autc$, let $\tau_{\sigma}$ be a square-integrable automorphic representation such that $\tau_{\sigma,f} \cong {}^\sigma\pi_f$. Then the A-parameter of $\tau_{\sigma}$ is 
\[  {}^\sigma\Psi := {}^\sigma\Sigma_1\stackrel{T}{\boxplus} \cdots \stackrel{T}{\boxplus}  {}^\sigma\Sigma_k.\] 
 
\vskip 5pt

\noindent  (ii) For each infinite place $v$, the $L$-parameter ${}^\sigma \Psi_v$ of $\tau_{\sigma,v}$   is  equal to
\[  \Psi_{\sigma^{-1}v} = \Pi_{1, \sigma^{-1} v} \boxplus \cdots \boxplus \Pi_{k, \sigma^{-1} v}. \] 
 In particular, $\tau_{\sigma,v}$ is a discrete series representation for each infinite place $v$ and 
 thus $\tau_{\sigma}$ is cuspidal. 
 \end{thm}

\begin{proof}
Again, we shall treat only symplectic and orthogonal groups in the proof; the case of unitary groups is similar. 
\vskip 5pt

\noindent   (i) Let us first check that ${}^\sigma \Psi$ is actually an $A$-parameter for $G$.  
   By Corollary \ref{cor:coho-classical} and Theorem \ref{thm:clozel},  one has the cuspidal automorphic representations ${}^\sigma \Sigma_i$ for each $i$. By Remark \ref{r:pole} and its analog for $\U(n)$ (cf. \ref{thm:arithmeticity-unitary}), ${}^\sigma \Sigma_i$ has the same symmetry type as $\Sigma_i$ (as detected by the poles of the exterior square, symmetric square or the Asai L-function in the respective cases) and hence as $\Psi$. Thus ${}^\sigma \Psi$ is a bona-fide $A$-parameter for $G$ and has an associated near equivalence class of square-integrable automorphic representations of $G$. 
   \vskip 5pt
   To prove (i), we need to show that the representation $\tau_{\sigma}$ is contained in the near equivalence class associated to ${}^\sigma \Psi$. This follows immediately by  Lemma \ref{lemma:unram}, since 
   \[  {}^\sigma \Pi_v \ \cong \  \ \stackrel{T}{\boxplus}_i {}^\sigma \Sigma_{i,v}. \]
   \vskip 5pt
   
   \noindent (ii) In the context of the theorem, one knows by Proposition \ref{prop:arch} that 
   \[  {}^\sigma \Psi_v \ = \ \ 
   \stackrel{T}{\boxplus}_i {}^\sigma \Sigma_{i,v} \ = \ \  
   \stackrel{T}{\boxplus}_i \Sigma_{i, \sigma^{-1}v} \ = \ \  
   \boxplus_i \Pi_{i,\sigma^{-1}v}. \]
   To be more precise, we have used Proposition \ref{prop:arch} to deduce that ${}^\sigma \Sigma_{i,v} = \Sigma_{i,\sigma^{-1}v}$. This follows immediately from Proposition \ref{prop:arch} if $n_i$ is even or if $F_v = \C$. 
    Thus, the only issue is when $n_i$  is odd and $F_v = \R$. Such a situation  only occurs when $G = \Sp(2n)$. But it follows by Proposition \ref{prop:cohom} (or rather the last paragraph of its proof) that there is a unique $i_0$ such that $n_{i_0}$ is odd. However, for this particular $i_0$, one has $(-1)^n \cdot \epsilon(\Pi_{i_0,v}) = 1$ for each infinite place $v$, where $\epsilon(\Pi_{i_0,v})$ is the sign in Proposition \ref{prop:arch}. In particular,  $\epsilon(\Pi_{i_0,v})$ is independent of $v$, so that ${}^\sigma \Pi_{i_0,v} = \Pi_{i_0, \sigma^{-1}v}$. 
  \vskip 5pt
  
  Finally, since the archimedean component ${}^\sigma \Psi_v = \Psi_{\sigma^{-1}v}$ is
  a discrete series $L$-parameter for $G_v$ and is the $L$-parameter of $\tau_{\sigma,v}$, we deduce that $\tau_{\sigma,v}$ is a discrete series representation. Since $\tau_{\sigma}$ is a square-integrable automorphic representation, it then follows by a well-known result of Wallach \cite{wallach} that $\tau_{\sigma}$ is cuspidal. This proves (ii).
     \end{proof} 
          \vskip 5pt
 
 \begin{cor}  \label{cor:stable}
 In the context of the theorem, suppose that $\Psi = \Pi$ is a cuspidal representation of $\GL(N)$. 
 Then ${}^\sigma\pi : =( \otimes_{v \in S_{\infty}} \pi_{\sigma^{-1}v}) \otimes {}^\sigma \pi_f$ is a cohomological cuspidal automorphic representation of $G$.
 \end{cor}
 
 \begin{proof}
 Under the hypothesis of the corollary, the $A$-packet associated to $\Psi$ is stable, in the sense that every member of the abstract global $A$-packet is automorphic. In this case, one may replace $\tau_{\sigma,\infty}$ by ${}^\sigma\pi_{\infty}$ and thus take $\tau_{\sigma}$ to be ${}^\sigma\pi$. 
 \end{proof}

    \vskip 10pt

\section{Arithmeticity of periods for classical groups} 
\label{sec:whittaker-classical}

After Theorem \ref{thm:classical}, it makes sense to consider the question of arithmeticity of periods for classical groups. We consider two examples here.

\vskip 5pt

\subsection{\bf Whittaker periods.}
 
 For a quasi-split group $G$ over $F$,
 if we fix a maximal $F$-torus $T$ contained in a Borel subgroup $B = T \cdot N$ over $F$, then 
 for any generic automorphic character $\psi$ of $N(\A_F)$, we may  
 consider the $\psi$-Whittaker period of an automorphic representation $\pi$ of $G$. If this period is nonzero on $\pi$, we say that $\pi$ is $\psi$-generic.
 \vskip 5pt
 
  The group $T(F)$ acts naturally on the set of generic automorphic charcaters of $N(\A_F)$ and the notion of being $\psi$-generic depends only on the $T(F)$-orbit of $\psi$. In particular, if $T(F)$ acts transitively on the set of generic automorphic characters of $N(\A_F)$, then there is no harm in suppressing $\psi$. This is the case when $G = \SO(2n+1)$ and $\U(2n+1)$. When $G = \Sp(2n)$, the set of $T(F)$-orbits of generic automorphic characters is a torsor of $F^{\times}/ F^{\times 2}$, whereas if $G = \U(2n)$, it is a torsor of $F^{\times}/ \mathbb{N}_{E/F} E^{\times}$. 
 
   \vskip 5pt
Now we have:
 
\begin{thm}
\label{thm:whittaker-orthogonal}
Let $F$ be a totally real number field, and let $G = \SO(2n+1)$ or $\U(2n+1)$.  Let $\pi$ be a cohomological cuspidal automorphic representation of $G(\A_F)$ which is globally generic. 
Then for any $\sigma \in \autc$, the conjugated representation 
${}^{\sigma}\pi$ (as defined in Corollary \ref{cor:stable}) is also a cohomological cuspidal automorphic representation of $G(\A_F)$ which is globally generic. 
\end{thm}

\begin{proof}
Since $\pi$ is globally generic, it follows by \cite{arthur-book} and \cite{ckpss} that the $A$-parameter $\Psi$ of $\pi$ is tempered. 
By Theorem \ref{thm:classical}, we know that ${}^\sigma\pi$ belongs to the global $A$-packet 
associated to the tempered parameter ${}^\sigma\Psi$. Moreover, ${}^\sigma\pi_v$ is locally ${}^\sigma \psi_v$-generic for all finite places $v$. But since all generic characters of $N(F_v)$ are in the same $T(F_v)$-orbit, we deduce that ${}^\sigma \pi_v$ is $\psi_v$-generic as well. The same holds at the infinite places, since ${}^\sigma \pi_v = \pi_{\sigma^{-1} v}$. Thus, we see that ${}^\sigma \pi$ is abstractly $\psi$-generic.
\vskip 5pt

Now in a local $L$-packet of $G(F_v)$, there can be at most one $\psi_v$-generic representation; this is a consequence of the theory of local descent (see Jiang-Soudry \cite{jiangs} for the case of $G = \SO(2n+1)$). 
Thus, ${}^\sigma\pi$ is the only member of its $A$-packet which could be globally $\psi$-generic. However, the theory of global descent  \cite{grs} says that a tempered $A$-packet must contain a globally $\psi$-generic cuspidal automorphic representation. Thus we conclude that 
${}^\sigma\pi$ is a globally generic cohomological cuspidal representation.
\end{proof}

\vskip 5pt
\begin{rem}{\rm 
When $G = \Sp(2n)$ or $\U(2n)$, it is still true that the representation ${}^\sigma \pi$ is abstractly generic with respect to the generic character $(\otimes_{v \in S_{\infty}} \psi_{\sigma^{-1}v}) \otimes  
(\otimes_{v \notin S_{\infty}} {}^\sigma \psi_v )$. However, we do not know whether the $T(\A_F)$-orbit of this generic character contains an automorphic character of $N(\A_F)$. Nevertheless, if $\sigma \in {\rm Aut}(\C/ \Q^{\rm ab})$, then ${}^\sigma \psi_v = \psi_v$ for all finite $v$, and so we conclude as in the theorem that ${}^\sigma \pi$ is globally $\psi$-generic again. One may avoid this complication by working with similitude groups, for example ${\rm GSp}(2n)$; however, one needs to await the generalization of Arthur's results \cite{arthur-book} to the context of similitude groups.
}\end{rem}
  
\vskip 5pt
\subsection{\bf Gross-Prasad period.}
In this speculative final subsection, we consider the Gross-Prasad periods for the classical groups. To be concrete, let us consider the Gross-Prasad period for unitary groups. 
 Let $\pi = \pi_1 \boxtimes \pi_2$ be a tempered cuspidal representation of $G= \U(n) \times \U(n-1)$.
Then a recent preprint of Wei Zhang \cite{zhang} establishes the global Gross-Prasad conjecture  under some local hypotheses.  
In particular, he shows that the period of $\pi$ over the diagonally embedded $\U(n-1)$ is nonzero if and only if 
 \[  L_E(\tfrac12, \Pi_1\times \Pi_2)  \ne 0. \]
 Here, $\Pi_i$ denotes the transfer of $\pi$ to $\GL(n)$ or $\GL(n-1)$ over $E$,  
 \vskip 5pt

Assume now that $\pi$  is stable and cohomological,  say $\pi \in {\rm Coh}(G,\mu)$. Suppose that $\pi$  has a nonvanishing period over the diagonally embedded $\U(n-1)$, and 
that $\Hom_{U(n-1)}(\mu, \C) \ne 0$. Then by Theorem \ref{thm:classical} and its corollary, one knows that ${}^\sigma\pi$ is also a cohomological cuspidal automorphic representation of $\U(n) \times \U(n-1)$. 
Now  one may apply the same argument as in the proof of Theorem \ref{thm:ggp-gln} (with the hypotheses stated there) to deduce that ${}^\sigma\pi$  also has nonvanishing period over the diagonally embedded $\U(n-1).$  We will perhaps leave the detailed treatment of this to a future occasion.

\end{document}